\documentclass[10pt,a4paper]{amsart}
\usepackage[utf8]{inputenc}
\usepackage{amsmath,amssymb,amsthm,enumerate,graphicx}
\usepackage{hyperref} % prolinkuje odkazy na clanky
\usepackage[shortlabels]{enumitem}
\usepackage{hyperref}
\newtheorem{theorem}{Theorem}[section]
\newtheorem{thmx}{Theorem}

\newtheorem{proposition}[theorem]{Proposition}
\newtheorem{corollary}[theorem]{Corollary}

\newtheorem{lemma}[theorem]{Lemma}
\newtheorem{fact}[theorem]{Fact}
\theoremstyle{definition}
\newtheorem{definition}[theorem]{Definition}
\newtheorem{notation}[theorem]{Notation}
\newtheorem*{thm*}{Theorem}

\newtheorem*{claim*}{Claim}

\newtheorem{remark}[theorem]{Remark}

\newcounter{que}
\newtheorem{question}[que]{Question}
\newcommand{\R}{\mathbb R}
\newcommand{\N}{\mathbb N}

\newcommand{\Span}{\operatorname{span}}

\newcommand\closedSpan[1]{\overline{\Span} \{#1\}}
\def\F{\mathcal{F}}
\def\AA{\mathcal{A}}
\def\BB{\mathcal{B}}

\def\NN{\mathcal{N}}
\def\MM{\mathcal{M}}
\def\PP{\mathcal{P}}
\def\TT{\mathcal{T}}
\def\bb{\mathfrak{b}}
\newcommand{\amen}{\mathfrak a}
\newcommand{\absamen}{\mathfrak{aa}}

\def\T{\mathcal{T}}

\newcommand{\Rea}{\mathbb{R}}

\newcommand{\Nat}{\mathbb{N}}

\newcommand{\co}{\operatorname{co}}

\newcommand{\desc}{\operatorname{desc}}
\newcommand{\pred}{\operatorname{pred}}
\newcommand{\rank}{\operatorname{rank}}
\newcommand{\leaf}{\operatorname{leaf}}
\newcommand{\Lip}{\operatorname{Lip}}

\newcommand\abs[1]{\mathopen|#1\mathclose|}

\newcommand\absb[2]{\csname#1l\endcsname|#2\csname#1r\endcsname|}
\newcommand\norm[1]{\mathopen\|#1\mathclose\|}

\newcommand\normb[2]{\csname#1l\endcsname\|#2\csname#1r\endcsname\|}

\begin{document}

\title[Canonical embedding of Lipschitz-free $p$-spaces]{Canonical embedding of Lipschitz-free $p$-spaces}

\author[M. C\' uth]{Marek C\'uth}
\author[T. Raunig]{Tom\'a\v{s} Raunig}
\email{cuth@karlin.mff.cuni.cz}
\email{raunig@karlin.mff.cuni.cz}

\address[M.~C\' uth, T.~Raunig]{Charles University, Faculty of Mathematics and Physics, Department of Mathematical Analysis, Sokolovsk\'a 83, 186 75 Prague 8, Czech Republic}

\subjclass[2020] {46A16, 46B80, 46B85}

\keywords{Quasi-Banach space, Lipschitz free $p$-space, $p$-amenability}
\thanks{Both authors were  supported by the GA\v{C}R project 23-04776S}

\begin{abstract}
We find a new finite algorithm for evaluation of Lipschitz-free $p$-space norm in finite-dimensional Lipschitz-free $p$-spaces. We use this algorithm to deal with the problem of whether given $p$-metric spaces $\NN\subset \MM$, the canonical embedding of $\F_p(\NN)$ into $\F_p(\MM)$ is an isomorphism. The most significant result in this direction is that the answer is positive if $\NN\subset \MM$ are metric spaces.
\end{abstract}
\maketitle

\section*{Introduction}

Lipschitz-free spaces form nowadays arguably one of the most important classes of Banach spaces having impact on various research areas, see e.g. \cite{AGPP, FG, HR, KV} for some recent results in metric geometry and (nonlinear) geometry of Banach spaces influenced by the study of Lipschitz-free spaces. There is also a  nonlocally convex analogue of this class of spaces, namely the Lipschitz-free $p$-spaces over $p$-metric spaces for $p\in (0,1]$, denoted by $\F_p(\MM)$, where the case of $p=1$ corresponds to the classical case of Lipschitz-free spaces. The study of Lipschitz-free $p$-spaces was initiated by N. Kalton and F. Albiac in \cite{AK09}, where spaces $\F_p(\MM)$ were applied to construct for each $p\in (0,1)$ an example of two separable $p$-Banach spaces which are Lipschitz isomorphic, but not linearly isomorphic (for $p=1$ this is nowadays one of the most important open problems in the nonlinear geometry of Banach spaces, see e.g. \cite[Problem 14.3.1]{AlbiacKaltonBook}). One decade after Lipschitz-free $p$-spaces were defined, in \cite{AACD} the authors initiated their systematic study and soon after several further papers dealing with Lipschitz-free $p$-spaces occurred, see e.g. \cite{BimaOld} for a recent contribution. It is not a surprise that some features of $\F_p(\MM)$ spaces are preserved from the case of $p=1$, while some others simply do not hold for $p<1$. Concerning the case when analogies coming from the case of $p=1$ do hold, note that an important feature of the case of $p<1$ is that we lack duality techniques which are naturally used in the case of $p=1$ and therefore it is often the case that when dealing with $p<1$ one is forced to develop new techniques which have impact also on the classical case when $p=1$, see e.g. \cite{AACD21}. On the other hand, sometimes even the very basic properties of Lispchitz-free spaces fail for the case when $p<1$. A basic evidence is that for any metric spaces $\NN\subset \MM$, it is straightforward to show that the canonical embedding of $\F_1(\NN)$ into $\F_1(\MM)$ is an isometry, while for $p<1$ this is not the case and it is even an open
question whether in general the canonical embedding is an isomorphism, see \cite[Theorem 6.1 and Question 6.2]{AACD}, respectively.

One of the main results contained in this paper is the following.

\begin{thmx}[see Theorem~\ref{thm:metricSolution}]\label{thmx:main}Given metric spaces $\NN\subset\MM$, the canonical embedding of $\F_p(\NN)$ into $\F_p(\MM)$ is an isomorphism for every $p\in(0,1)$.
\end{thmx}
\noindent We admit that this does not entirely solve \cite[Question 6.2]{AACD} as we do not know whether Theorem~\ref{thmx:main} holds if we assume that $\NN\subset\MM$ are $p$-metric spaces for $p<1$. Still, it solves quite an important special case and implies e.g. positive answer to \cite[Question 6.4]{AACD20}, see Proposition~\ref{prop:nets} below. Moreover, our method of proof does say at lest something even for the general case.

The starting point of our investigation is a simple and instructive formula for the computation of a Lipschitz-free $p$-norm, see Theorem~\ref{thm:algorithm}. This could be new even for the case of $p=1$, even though in this case we do not have any application (we refer the reader to the beginning of Subsection~\ref{subsec:applicAlgorithm} for a discussion concerning importance of Theorem~\ref{thm:algorithm} for the case of $p=1$). The first evidence that our algorithm for the computation Lipschitz-free $p$-space norm might be useful is described in Remark~\ref{rem:amenability}; namely, using Theorem~\ref{thm:algorithm} we can quite substantially simplify the proof of \cite[Theorem 6.1]{AACD}, where the authors found the first example of metric spaces $\NN\subset\MM$ such that the canonical embedding of $\F_p(\NN)$ into $\F_p(\MM)$ is not an isometry. Encouraged by this observation, we work further in this direction in Section~\ref{sec:pAmen}. Apart from the above mentioned Theorem~\ref{thmx:main} we prove Theorem~\ref{thm:isomAmen}, where we give a characterization of those $p$-metric spaces $\NN$ with $|\NN|=3$ for which the canonical embedding from $\F_p(\NN)$ into $\F_p(\MM)$ is isometry whenever $\MM$ is a~$p$-metric space with $\NN\subset \MM$. We also prove Theorem~\ref{thm:reductionExamples}, where we find out that when dealing with \cite[Question 6.2]{AACD} it suffices to consider a specific class of $p$-metric spaces coming from weighted trees and their subspaces given by leaves in those trees - this is the most important step in the proof of Theorem~\ref{thmx:main} and also it enables us to come up with several other estimates summarized in Subsection~\ref{subsec:estimates}.

The structure of the paper is the following. In Section~\ref{sec:prelim} we collect preliminaries and notation including the definition of Lipschitz-free $p$-spaces and their basic properties. The main aim of Section~\ref{sec:algo} is to prove Theorem~\ref{thm:algorithm} giving us an algorithm for the computation of a Lipschitz-free $p$-space norm. In Section~\ref{sec:pAmen} we develop applications of Theorem~\ref{thm:algorithm} connected to \cite[Question 6.2]{AACD} including the proof of Theorem~\ref{thmx:main}. There are still many natural questions, which we leave open and some of those are collected in Section~\ref{sec:open}.

\section{Preliminaries and notation}\label{sec:prelim}

Given two sets $A\subset B$ we naturally identify $\Rea^A$ with the subsest of $\Rea^B$ defined as $\{x\in \R^B\colon x|_{B\setminus A}\equiv 0\}$.

Let $p\in (0,1]$. We say $(\MM,d)$ is a $p$-metric space if $(\MM,d^p)$ is a metric space. By a \emph{nontrivial $p$-metric space} we understand a metric space consisting of at least two points. We say $(\MM, d, 0)$ is a~pointed $p$-metric space if $(\MM, d)$ is a~$p$-metric space and $0 \in \MM$ is an arbitrary distinguished point.

All the vector spaces in this paper are over the field of real numbers. A $p$-Banach space is a vector space $X$ endowed with a $p$-norm $\norm{\cdot}_X:X\to [0,\infty)$ (a~mapping satisfying all the axioms of the norm except that the triangle inequality is replaced with the $p$-triangle inequality, that is, $\norm{x+y}^p\leq \norm{x}^p + \norm{y}^p$) and moreover $(X,\norm{\cdot}_X)$ is complete. For some preliminaries concerning $p$-metric spaces and $p$-Banach spaces we refer the interested reader to \cite[Section 2]{AACD}. Let $X$ be a~$p$-Banach space. Given a~finite set $F$ and $x\in X^F$ it is obvious what we mean by the expression $\sum_{f\in F} x_f$, having in mind that if $F=\emptyset$ we put $\sum_{f\in F} x_f:=0$. A~set $A\subset X$ is \emph{$p$-convex} if $\lambda x + \mu y\in A$ whenever
$x,y \in A$ and $\lambda,\mu\geq 0$ are such that $\lambda^p + \mu^p = 1$. Given $A\subset X$, the \emph{$p$-convex hull} of $A$, $\co_p(A)$ for short, is the smallest $p$-convex set containing $A$. Note that if $A$ is a symmetric set containing the origin, for $p<1$ the $p$-convex hull and ``absolutely $p$-convex hull'' are equivalent notions, see e.g. \cite[Lemma 1]{BBP}. 

Let $(\MM,d,0)$ be a pointed $p$-metric space. By $\Rea^\MM_0$ we denote the vector space of (not necessarily continuous) mappings $f:\MM\to \Rea$ satisfying $f(0)=0$. For $x\in \MM$ we denote by $\delta_\MM(x)$ (or just $\delta(x)$ for short) the linear mapping $\delta_\MM(x):\Rea^\MM_0\to \Rea$ given by $\delta_\MM(x)(f):=f(x)$. Finally, the \emph{Lipschitz-free $p$-space over $\MM$}, denoted by $\F_p(\MM)$, is the completion of the linear span $\PP(\MM) = \Span\{\delta_\MM(x)\colon x\in\MM\}$ equipped with the $p$-norm
\[
\normb{Big}{\sum_{i=1}^n a_i\delta(x_i)}:=\sup \normb{Big}{\sum_{i=1}^n a_if(x_i)}_Y,
\]
where the supremum is taken over all $p$-Banach spaces $(Y,\norm{\cdot}_Y)$ and all choices of
$1$-Lipschitz maps $f:\MM\to Y$ with $f(0) = 0$. We refer the interested reader to \cite[Section 4]{AACD} for basic properties of Lipschitz-free $p$-spaces. The ones which are the most relevant for this paper are summarized below. The mapping $\delta_\MM:\MM\to\F_p(\MM)$ is isometry and given two pointed $p$-metric spaces $\NN$ and $\MM$ and a Lipschitz map $f:\NN\to \MM$ with $f(0_\NN)=0_\MM$, there exists unique linear operator $L_f:\F_p(\NN)\to \F_p(\MM)$ such that $\norm{L_f} = \Lip(f)$ and $L_f\circ \delta_\NN = \delta_\MM\circ f$, we say that $L_f$ is the \emph{canonical linearization of $f$}. 
Further, we use the following notation
\[\begin{split}
\AA(\MM) = \Big\{\frac{\delta(x)-\delta(y)}{d(x,y)}\colon x\neq y\in\MM\Big\}\subset \F_p(\MM).
\end{split}\]
The following is an easy consequence of the fact that for $\PP(\MM)\cap B_{\F_p(\MM)} = \co_p(\AA(\MM))$, see \cite[Definition 4.9 and Theorem 4.10]{AACD}.
\begin{fact}\label{fact:wlogFinite}Let $p\in(0,1]$ and $(\MM,d,0)$ be a pointed $p$-metric space. Then for a finite set $0\in F\subset \MM$ and $a\in\R^F$ we have
\[
\normb{Big}{\sum_{f\in F}a_f \delta(f)}_{\F_p(\MM)} = \inf\Big\{\normb{Big}{\sum_{f\in F}a_f \delta(f)}_{\F_p(\NN)}\colon \NN\subset\MM\text{ is finite and }F\subset \NN\Big\}.
\]
\end{fact}

\section{Algorithm for the computation of the Lipschitz-free \texorpdfstring{$p$-}-norm}\label{sec:algo}

The aim of this section is to prove Theorem~\ref{thm:algorithm}, which provides us with a new finite algorithm for the computation of a Lipschitz-free $p$-norm. This is based on building trees with vertices in the corresponding $p$-metric space. Before stating the main result, let us establish some notation.

\begin{definition}\label{def:trees}
Let $V$ be a finite set with a distinguished point $0\in V$. By $\TT(V,0)$ we denote the set of all the rooted trees $T = (V,E,0)$ with vertices $V$, edges $E$ and distinguished point $0\in V$, we say $0$ is the \emph{root} of $T$. Recall that a \emph{path (from $v_0$ to $v_n$)} in a graph $G = (V,E)$ is one-to-one sequence of vertices $v_0,\ldots,v_n$ such that $\{v_i,v_{i+1}\}\in E$ for $0\leq i\leq n-1$. Given $T = (V,E,0)\in \TT(V,0)$ and $x\in V$, we denote
\begin{itemize}    
    \item $\desc(x)$ the immediate descendants of $x$, that is vertices $v\in V$ such that there is a path $v_0,\ldots,v_n$ from $0$ to $v$ in $T$ with $v_{n-1} = x$;
    \item if $x\neq 0$, then $\pred(x)$ denotes the unique vertex $v\in V$ such that if $v_0,\ldots,v_n$ is the unique path from $0$ to $x$ in $T$ then $v_{n-1}=v$; we denote $e_x:=\{\pred(x),x\}$;
    \item by $\rank(x)\in \N_0$ we denote the length of the path from the root $0$ to the vertex $x$, that is, for $x=0$ we put $\rank(x)=0$ and if $x\neq 0$, there is a unique path $v_0,\ldots,v_n$ from $0$ to $x$ and we put $\rank(x) = n$;
    \item $\rank(T):=\max \{\rank(x)\colon x\in V\}$;
    \item \emph{leaves of $T$} are vertices from the set $\leaf(T):=\{x\in V\colon \desc(x)=\emptyset\}$;
    \item by $V_x$ we denote vertices in the maximal connected subgraph of $(V,E\setminus\{e_x\})$ containing $x$ as a vertex (that is, vertices in the subtree of $T$ ``rooted at the point $x$'').
\end{itemize}
Finally, note that there is a natural orientation of edges on any rooted tree $T$ given by oriented edges $\{(\pred_T(x),x)\colon x\in \MM\setminus\{0\}\}$. Thus, when it is convenient we work with $T$ as with an oriented tree and denote edges as $(x,y)$ instead of $\{x,y\}$, meaning by this that $x = \pred(y)$.\\

\medskip

\hspace{100pt}\includegraphics[scale=0.7]{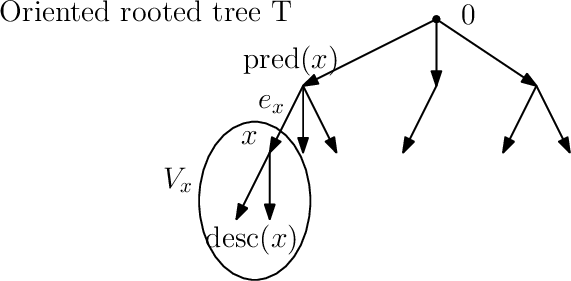}\\
If we want to emphasize to which $T\in \TT(V,0)$ the notions correspond, we write $E^T$, $\desc_T(x)$, $\pred_T(x)$, $\rank_T(x)$, $e_x^T$, and $V_x^T$ instead of $E$, $\desc(x)$, $\pred(x)$, $\rank(x)$, $e_x$ and $V_x$, respectively.

If $p\in (0,1]$ and $(\MM,d,0)$ is a nontrivial pointed $p$-metric space, we write $\TT(\MM)$ instead of $\TT(\MM,0)$. For $T=(\MM,E,0)\in\TT(\MM)$ and $e=\{x,y\}\in E$ we put $d(e):=d(x,y)$.
\end{definition}

The formula which we discovered is the following. It can be viewed as a finite algorithm for the computation of a Lipschitz-free $p$-norm with $|\TT(\MM)|$ steps.

\begin{theorem}\label{thm:algorithm}
Let $(\MM,d,0)$ be a nontrivial finite pointed $p$-metric space for some $p\in(0,1]$. Then for any $a\in \R^{\MM\setminus\{0\}}$ we have 
\[
\normb{Big}{\sum_{x\in \MM\setminus\{0\}} a_x\delta(x)}_{\F_p(\MM)}^p = \min \Big\{ \sum_{x\in \MM\setminus\{0\}} \absb{Big}{\Big(\sum_{y\in V_x^T} a_y\Big)d(e_x^T)}^p\colon T\in \TT(\MM)\Big\}.
\]
\end{theorem}

Since we will use our algorithm in what follows, we use the following notation.

\begin{notation}Let $(\MM,d,0)$ be a nontrivial finite pointed $p$-metric space for some $p\in(0,1]$. Then for any $a\in \R^{\MM}$, $x\in \MM$ and $T\in \TT(\MM)$ we put
\[
c_T(x,a) := \sum_{y\in V_x^T} a_y\qquad \text{and}\qquad
T(a):=\Big(\sum_{x\in \MM\setminus\{0\}} \absb{Big}{c_T(x,a)d(e_x^T)}^p\Big)^{1/p}.\]
If we want to emphasize the $p$-metric on $\MM$, we write $T_{d}(a)$ instead of $T(a)$.
\end{notation}

Subsection~\ref{subsec:proofAlgorithm} is devoted to the proof of Theorem~\ref{thm:algorithm}. In Subsection~\ref{subsec:applicAlgorithm} we discuss the importance of the formula from Theorem~\ref{thm:algorithm} for the case of $p=1$ and some applications for the case of $p<1$, respectively.

\subsection{Proof of Theorem~\ref{thm:algorithm}}\label{subsec:proofAlgorithm} We begin with the following observation, which we shall later use for $Z = \AA(\MM)\subset \F_p(\MM)$.

\begin{lemma}\label{lem:pConvexHullBall}Let $X\neq \{0\}$ be a finite-dimensional $p$-Banach space for some $p\in(0,1]$ and $Z\subset B_X$ be a symmetric set such that $\co_p(Z) = B_X$. Let us denote by $\BB$ the set of all the algebraic bases of $X$ consisting of points from $Z$. Then for $x\in X$ we have
\begin{equation}\label{eq:formulaWithBasis}
\norm{x} = \min\Big\{\Big(\sum_{b\in \bb} \abs{a_b}^p\Big)^{1/p}\colon \bb\in\BB, a\in\R^\bb \text{ is such that } x = \sum_{b\in\bb} a_b b\Big\}.
\end{equation}
\end{lemma}
\begin{proof}
Observe that for any $\bb\in\BB$ and $a\in \R^\bb$ satisfying $x = \sum_{b\in\bb}a_{b} b$, using triangle $p$-inequality and the fact that $\bb\subset B_X$ we obtain $\|x\|^p\leq \sum_{b\in\bb}|a_{b}|^p$. Thus, it suffices to prove that for some $\bb\in\BB$ equality holds.

Pick $x\in X$, without loss of generality we can assume that $x\in S_X$. Since $S_X\subset \co_p(Z)$, using the Carath\'eodory theorem (for the proof in the case of $p<1$, see \cite[Theorem 1]{BBP}), there is a~linearly independent set $F\subset Z$ such that $x\in \co_p(F)$. Thus, there is $N\in\N$, $\lambda \in (0,1]^N$ and $v\in F^N$ such that $x = \sum_{i=1}^N \lambda_i v_i$ and $\sum_{i=1}^N \lambda_i^p = 1$.

We \emph{claim} that we may without loss of generality assume that $v_i\neq v_j$ for $i\neq j$. Indeed, if $p<1$ then supposing that $v_j = v_k$ for some $j\neq k$ we obtain
\[
1 = \norm{x}^p = \normb{Big}{(\lambda_j + \lambda_k)v_j + \sum_{i=1, i\neq j,k}^N \lambda_i v_i}^p \leq (\lambda_j + \lambda_k)^p + \sum_{i=1, i\neq j,k}^N \lambda_i^p < \sum_{i=1}^N \lambda_i^p = 1,
\]
a contradiction. On the other hand, if $p=1$ then we may put together the coefficients $\lambda_j$ corresponding to the same $v_i$. More precisely, if $J = \{i\geq 2\colon v_i = v_1\}$ is nonempty, we put $\lambda'_1:=\lambda_1 + \sum_{i\in J} \lambda_i$ and $\lambda'_i:=\lambda_i$ for $i\notin J\cup\{1\}$, then $u=\sum_{i\in\{1,\ldots,N\}\setminus J} \lambda'_i v_i$ and $\sum_{i\in\{1,\ldots,N\}\setminus J} \lambda'_i = 1$, so we may assume that $v_1\neq v_j$ for $j\geq 2$ and we proceed inductively. This proves the claim.

Thus, there is $\lambda\in [0,1]^F$ such that $x = \sum_{v\in F}\lambda_v v$ and $\sum_{v\in F} \lambda_v^p = 1$. Extending $F$ to a basis $\bb\in \BB$ and $\lambda$ to $a\in [0,1]^\bb$ by putting $\lambda_b = 0$ for $b\in\bb\setminus F$ we obtain $x = \sum_{b\in \bb} a_b b$ and $\|x\|^p = 1 = \sum_{b\in \bb} |a_b|^p$.
\end{proof}

The following aims at identifying all the basis of $\F_p(\MM)$ consisting of points from $Z= \AA(\MM)\subset \F_p(\MM)$. Similar observation was done also e.g. in \cite[Lemma 2.5]{AFGZ}, we provide here the reader with the short argument for the sake of completeness.

\begin{lemma}\label{lem:allBases}
Let $n\geq 1$ and $\{e_1,\ldots,e_n\}$ be a basis in a vector space $X$, $\{c_{i,j}\colon i,j=0,\ldots,n,\, i\neq j\}\subset \Rea\setminus\{0\}$, put $e_0:=0$ and consider the set
\[
W:= \Big\{\frac{e_i-e_j}{c_{i,j}}\colon i,j = 0,\ldots,n,\, i\neq j\Big\}.
\]
Then all the bases of $X$ consisting of elements from $W$ are exactly
\begin{equation}\label{eq:canonicalBasis}
\mathfrak{b}(T):=\Big\{\frac{e_i-e_j}{c_{i,j}}\colon \{e_i,e_j\}\in E\Big\},
\end{equation}
where $T = (V,E,0)\in\TT(V,0)$ and $V = \{e_i\colon i=0,\ldots,n\}$.
\end{lemma}
\begin{proof}First, since the linear independence does not depend on the choice of $c_{i,j}$, we may and do without loss of generality assume that $c_{i,j}=1$ for every $i,j=0,\ldots n$ with $i\neq j$.
It is easy to observe that for each $T\in\TT(V,0)$, $\mathfrak{b}(T)$ is a basis of $X$. On the other hand, if $\mathfrak{b}\subset W$ is a basis of $X$, the graph $G = (V,E)$ with edges $E = \{\{e_i,e_j\}\colon e_i-e_j\in \mathfrak{b}\}$ does not contain a cycle due to linear independence of $\mathfrak{b}$. Moreover, we have $|V| - 1 = |\mathfrak{b}| = |E|$ which implies that $G$ is connected (in general, for an acyclic graph we have $|E| = |V|-k$, where $k$ is the number of connected components). Thus, $G$ is a tree which easily implies that we have $\mathfrak{b} = \mathfrak{b}(T)$ for $T = (V,E,0)$.
\end{proof}

Our final step towards the proof of Theorem~\ref{thm:algorithm} is to find out what are the coefficients of a point in $\F_p(\MM)$ with respect to a basis consisting of elements from $\AA(\MM)$.

\begin{lemma}\label{lem:coefficients}Let $n\geq 1$ and $\{e_1,\ldots,e_n\}$ be a basis in a vector space $X$, $V = \{0,e_1,\ldots,e_n\}$ and $T = (V,E,0)\in \TT(V,0)$. Suppose that for some $a,b\in \Rea^n$ and $c\in(\Rea\setminus\{0\})^n$ we have
\[
\sum_{i=1}^n a_i e_i = \sum_{i=1}^n b_i \frac{e_i-\pred_T(e_i)}{c_i}.
\]
Then $b_i = c_i\sum_{\{j\colon e_j\in V_{e_i}^T\}} a_j$ for every $i=1,\ldots,n$.
\end{lemma}
\begin{proof}The proof is by induction with respect to $\dim X = n$. For $n=1$ this is trivial, because $\TT(V,0)$ contains only the tree $(V,\{e_1\},0)$ and so we have $a_1e_1 = b_1\tfrac{e_1-0}{c_1}$ which implies that $b_1 = c_1a_1$.

Suppose Lemma holds for $n$ and let us prove it for $n+1$. We may without loss of generality assume that $e_{n+1}\in \leaf(T)$ and $e_n = \pred_T(e_{n+1})$ (otherwise we just reenumerate points in the basis). Comparing the coefficients before $e_{n+1}$ we obtain that $a_{n+1} = \tfrac{b_{n+1}}{c_{n+1}}$ and therefore $b_{n+1} = c_{n+1}a_{n+1}$. Further, consider $V' = V\setminus\{e_{n+1}\}$ and the tree $T' = (V',E',0)\in \TT(V',0)$ with $E' = E\setminus\{\{e_n,e_{n+1}\}\}$. Then we have
\[\begin{split}
\sum_{i=1}^{n-1} a_ie_i & + (a_n + a_{n+1})e_n = \sum_{i=1}^{n} a_ie_i + \frac{b_{n+1}}{c_{n+1}}e_n \\ & = \sum_{i=1}^n b_i \frac{e_i-\pred_T(e_i)}{c_i} - \frac{b_{n+1}}{c_{n+1}}e_n + \frac{b_{n+1}}{c_{n+1}}e_n = \sum_{i=1}^n b_i \frac{e_i-\pred_{T'}(e_i)}{c_i}.
\end{split}\]
Thus, if we put $\widehat{a_i}:=a_i$ for $i\leq n-1$ and $\widehat{a_n}:=a_n + a_{n+1}$, by the inductive assumption for every $i\leq n$ we obtain
\[
b_i = c_i\sum_{\{j\colon e_j\in V_{e_i}^{T'}\}} \widehat{a_j} = c_i \sum_{\{j\colon e_j\in V_{e_i}^T\}} a_j.
\]
This finishes the inductive step.
\end{proof}

Let us finish now this subsection with the proof of Theorem~\ref{thm:algorithm}.

\begin{proof}[Proof of Theorem~\ref{thm:algorithm}]
Pick $a\in \R^{\MM\setminus\{0\}}$ and put $\mu = \sum_{x\in\MM\setminus\{0\}} a(x)\delta(x)$. Note that $\co_p(\AA(\MM)) = \overline{\co_p}(\AA(\MM)) = B_{\F_p(\MM)}$ (see \cite[Corollary 4.11]{AACD}), so by Lemma~\ref{lem:pConvexHullBall}, if we denote by $\BB$ all the basis of $\F_p(\MM)$ consisting of the points from $\AA(\MM)$, we obtain
\begin{equation}\label{eq:formulaFirst}
\norm{\mu}^p = \min\Big\{\sum_{b\in \bb} \abs{c_b}^p\colon \bb\in\BB, c\in\R^\bb \text{ is such that } \mu = \sum_{b\in\bb} c_b b\Big\}.
\end{equation}
By Lemma~\ref{lem:allBases}, we obtain that $\BB$ consists of $\{\tfrac{\delta(x)-\delta(y)}{d(x,y)}\colon \{x,y\}\in E^T\}$ for all the choices of $T\in\TT(\MM)$. Since in \eqref{eq:formulaFirst} there are absolute values of the coefficients $c_b$, we may consider in the formula \eqref{eq:formulaFirst} only minimum over basis $\bb$ of the form 
\[
\bb(T) := \{\tfrac{\delta(x)-\delta(\pred_T(x))}{d(x,\pred_T(x))}\colon x\in \MM\setminus\{0\}\} = \{\tfrac{\delta(x)-\delta(\pred_T(x))}{d(e_x^T)}\colon x\in \MM\setminus\{0\}\}
\]
for all the choices of $T\in\TT(\MM)$. Thus, $\|\mu\|^p$ is the minimum of all the numbers $\sum_{x\in\MM\setminus\{0\}} \abs{c_x}^p$, where $c\in\R^{\MM\setminus\{0\}}$ is such that
\begin{equation}\label{eq:formulaSecond}
\mu = \sum_{x\in \MM\setminus\{0\}} c_x  \frac{\delta(x)-\delta(\pred_T(x))}{d(e_x^T)}
\end{equation}
for some $T\in\TT(\MM)$. Finally, by Lemma~\ref{lem:coefficients} we observe that \eqref{eq:formulaSecond} implies
\[
c_x = d(e_x)\sum_{y\in V_x^T}a(y),\qquad x\in \MM\setminus\{0\}.
\]
Thus, we obtain the formula from the statement of Theorem~\ref{thm:algorithm}.
\end{proof}

\subsection{Importance of the formula from Theorem~\ref{thm:algorithm}}\label{subsec:applicAlgorithm}
First, let us discuss the importance of the algorithm from Theorem~\ref{thm:algorithm} for the case of $p=1$. The problem of finding the value of $\norm{\mu}_{\F_1(\MM)}$ for some $\mu\in\PP(\MM)$ is well-known to be a linear programming problem and for such there are well-known numerical algorithms, see e.g. \cite[Section 5.10]{MGbook} for the classical ``simplex method'', there are even (more difficult to understand) algorithms of polynomial-time, see e.g. \cite{W05}. Moreover, one can also use methods concerning the ``minimum cost maximum flow'' problem, which gives some more possibilities of an evaluation of the Lipschitz-free norm, see e.g. \cite{CCPRSbook}. On the other hand, it seems that all of the algorithms described above are quite involved when compared with our Theorem~\ref{thm:algorithm}, which still in some situations enables us to come up with reasonable formulas, see e.g. Corollary~\ref{cor:twoPointsFormula} and the comment below.

For $p<1$ the issue of finding the value of $\norm{\mu}_{\F_p(\MM)}$ for some $\mu\in\PP(\MM)$ is no more linear programming problem and the problem itself seems to be much more complicated. Even though the importance of Theorem~\ref{thm:algorithm} is most probably limited for researchers in numerical analysis (the algorithm from Theorem~\ref{thm:algorithm} is of exponential time), we shall see later that it enables us to substantially push forward our theoretical knowledge of Lipschitz-free $p$-spaces. Let us start with the following formula.

\begin{corollary}\label{cor:twoPointsFormula}Let $p\in (0,1]$ and $(\MM,d,0)$ be a $p$-metric space with $\MM=\{0,x,y\}$. Denote $d_x = d(x,0)$, $d_y=d(y,0)$ and $d_{xy}=d(x,y)$. Then for any $a,b\in\R$ we have
\begin{equation}\label{eq:twoPointsFormula}\begin{split}
\norm{a\delta(x) + b\delta(y)}_{\F_p(\MM)}^p = \min\{|ad_x|^p + |bd_y|^p, & |(a+b)d_x|^p + |bd_{xy}|^p,\\&  |(a+b)d_y|^p + |ad_{xy}|^p\}.
\end{split}\end{equation}
\end{corollary}
\begin{proof}We apply Theorem~\ref{thm:algorithm} and realize that $\TT(\MM) = \{(\MM,E_i,0)\colon i=1,2,3\}$, where $E_1:=\{(0,x),(0,y)\}$, $E_2:=\{(0,x),(x,y)\}$ and $E_3:=\{(0,y),(y,x)\}$. The first expression in \eqref{eq:twoPointsFormula} corresponds to $E_1$, the second to $E_2$ and the third to $E_3$.
\end{proof}

Corollary~\ref{cor:twoPointsFormula} may be seen as a generalization of \cite[Lemma 11]{CJ}, where basically the same formula is given for $p=1$. In the case of $p=1$ one can easily find out (and the authors from \cite{CJ} did) which of the three expression in \eqref{eq:twoPointsFormula} is the one where minimum is attained just based on the knowledge of the coefficients. This is far from being the case for $p<1$ as witnessed e.g. by Theorem~\ref{thm:positiveCoef}.

In the $p=1$ case it is well-known that given a finite metric space $\MM$ and coefficients $a\in[0,\infty)^\MM$, we have $\norm{\sum_{x\in\MM} a_x\delta(x)}_{\F(\MM)} = \sum a_xd(x,0)$. The following result shows that this is not true for $p<1$ and moreover it even characterizes those $p$-metric spaces, for which an analogy holds.

\begin{theorem}\label{thm:positiveCoef}
Let $0<p<1$, $(\MM,d,0)$ be a pointed $p$-metric space with $|\MM|\geq 3$. Then the following assertions are equivalent.
\begin{enumerate}[(a)]
    \item\label{it:positiveBehavesOk} For any $a\in [0,\infty)^\MM$ such that $\sum_{x\in\MM} \big(a_x d(x,0)\big)^p<\infty$, we have
    \[
    \norm{\sum_{x\in\MM} a_x\delta(x)}_{\F_p(\MM)}^p = \sum_{x\in\MM} \big(a_x d(x,0)\big)^p.
    \]
    \item\label{it:twoPositiveBehavesOk} For any $\alpha\geq 0$ and two distinct $x,y\in\MM\setminus\{0\}$ we have
     \[
     \norm{\delta(x) + \alpha\delta(y)}_{\F_p(\MM)}^p = d^p(x,0) + \alpha^pd^p(y,0).
     \]
     \item\label{it:conditionPositiveBehavesOk} For every $x\in\MM\setminus\{0\}$ we have $d(x,0) = d(x,\MM\setminus\{x\})$.
\end{enumerate}
\end{theorem}
\begin{proof}[Proof of Theorem~\ref{thm:positiveCoef}]
\ref{it:conditionPositiveBehavesOk}$\Rightarrow$\ref{it:positiveBehavesOk}: Note that it suffices to prove the assertion for the case when $F=\{x\in\MM\colon a_x\neq 0\}$ is a finite set, because in the general case we can then pass to the limit on both sides (limit of finite sums with respect to net indexed by finite subsets of $\MM$). Put $\mu = \sum_{x\in F}a_x\delta(x)$. By Fact~\ref{fact:wlogFinite}, we may without loss of generality assume that $\MM$ is a finite metric space. But then, for every $T\in\TT(\MM)$ using \ref{it:conditionPositiveBehavesOk} we obtain
\[
\sum_{x\in\MM\setminus\{0\}} \Big(d(e_x)\sum_{y\in  V_x} a_y\Big)^p\geq \sum_{x\in\MM\setminus\{0\}} \big(d(e_x)a_x\big)^p\geq \sum_{x\in\MM\setminus\{0\}} \big(d(x,0)a_x\big)^p,
\]
so by Theorem~\ref{thm:algorithm} we obtain $\norm{\mu}_{\F_p(\MM)}^p\geq \sum_{x\in\MM} \big(d(x,0)a_x\big)^p$ and the other inequality follows easily from the triangle $p$-inequality.\\
\ref{it:positiveBehavesOk}$\Rightarrow$\ref{it:twoPositiveBehavesOk}: is trivial.\\
\ref{it:twoPositiveBehavesOk}$\Rightarrow$\ref{it:conditionPositiveBehavesOk}: Suppose there are distinct point $x,y\in\MM\setminus\{0\}$ with $d(x,y)<d(x,0)$. Then we have $d(x,0)^p - d(x,y)^p>0$ and so, since $\lim_{\alpha\to\infty} (1+\alpha)^p - \alpha^p = 0$, we may find $\alpha>0$ such that
\[
\frac{d(x,0)^p - d(x,y)^p}{d(y,0)^p} > (1+\alpha)^p - \alpha^p.
\]
But then for $\mu = \delta(x) + \alpha\delta(y)\in \F_p(\MM)$, using \eqref{eq:twoPointsFormula}, we obtain 
\[
\norm{\mu}_{\F_p(\MM)}^p\leq \norm{\mu}_{\F_p(\{0,x,y\})}^p\leq (1+\alpha)^pd^p(y,0) + d^p(x,y) < d^p(x,0) + \alpha^p d^p(y,0),
\]
so \ref{it:twoPositiveBehavesOk} does not hold.
\end{proof}

\begin{remark}\label{rem:amenability}
Note that Theorem~\ref{thm:positiveCoef} is related to the example of a tuple of metric spaces $\NN\subset\MM$ such that the canonical embedding of $\F_p(\NN)$ into $\F_p(\MM)$ is not an isometry, see \cite[Theorem 6.1]{AACD}. Indeed, the most technical part of the proof of \cite[Theorem 6.1]{AACD} is contained in \cite[Proposition 4.15]{AACD}, which easily follows from our Theorem~\ref{thm:positiveCoef}, because for the very special case of a finite metric space $\MM = \{0,x_1,\ldots,x_d\}$ with distances $0$ and $1$, we obtain from Theorem~\ref{thm:positiveCoef} that the $p$-norm $\|\cdot\|_{(p)}$ on $\R^d$ given for $a\in\R^d$ by the formula $\|a\|_{(p)}:=\norm{\sum_{i=1}^n a_i\delta(x_i)}_{\F_p(\MM)}$ satisfies both (a) and (b) from \cite[Proposition 4.15]{AACD}.
\end{remark}

\section{\texorpdfstring{$p$-}-amenability}\label{sec:pAmen}

Following \cite[Definition 2.5]{AACD21}, we say that a subset $\NN$ of a $p$-metric space $\MM$ is \emph{$p$-amenable in
$\MM$ with constant $C <\infty$} if $L_j:\F_p(\NN)\to\F_p(\MM)$ is an isomorphism and 
$\norm{L_j^{-1}}\leq C$, where $L_j$ is the
canonical linearization of the inclusion map $j:\NN\to\MM$. If $C=1$, we say that $\NN$ is \emph{isometrically $p$-amenable in $\MM$}. Recall that for $p=1$ any subset of a metric space is isometrically $p$-amenable, while for $p<1$ there are known metric spaces and their subsets which are not isometrically $p$-amenable and it is open whether any subset of a $p$-metric space is $p$-amenable with some constant, see \cite[Section 6]{AACD}. As we mentioned in Remark~\ref{rem:amenability}, our Theorem~\ref{thm:algorithm} seems to provide us with a new tool which could be useful when dealing with $p$-amenability. The aim of this section is to explore it in a greater detail. We start with some preliminary general observations which shall be used in later subsections. The main outcomes of this section are the following: in Subsection~\ref{subsec:isometricAmen} we give a characterization of those pointed $p$-metric spaces consisting of $3$ points which are isometrically $p$-amenable in any superspace (Theorem~\ref{thm:isomAmen}); in Subsection~\ref{subsec:reductionToTrees} we prove that when dealing with an optimal constant of $p$-amenability, it suffices to consider $p$-metric spaces given by weighted trees and their subspaces given by leafes in this tree (Theorem~\ref{thm:reductionExamples}); in Subsection~\ref{subsec:estimates} we deal with several estimates of the constant of $p$-amenability including  Theorem~\ref{thm:metricSolution}, which is a quantitative refinement of Theorem~\ref{thmx:main} mentioned in the Introduction.

In order to formulate our results in a short and elegant way, we introduce the following quantities.

\begin{definition}Let $0<p\leq q \leq 1$, $(\MM,d,0)$ be a finite pointed $q$-metric space and $0\in \NN\subset \MM$ be its subspace and $n,k\in\N$, $n\leq k$. Then
\[\begin{split}
\amen_p(\NN,\MM):= & \min\{C\colon \NN\text{ is $p$-amenable in }\MM\text{ with constant }C\},\\
\absamen_p^q(\NN):= & \sup\{\amen_p(\NN,\MM')\colon \MM'\text{ is a $q$-metric space such that }\NN\subset\MM'\},\\
\absamen_p^q(n,k):= & \sup\{\amen_p(\NN,\MM)\colon \NN\subset \MM \text{ are $q$-metric spaces, }|\NN\setminus\{0\}|\leq n\\
& \qquad\qquad\qquad \text{ and }|\MM\setminus\{0\}|\leq k\},\\
\absamen_p^q(n):= & \sup\{\absamen_p^q(\NN)\colon \NN \text{ is a $q$-metric space and }|\NN\setminus\{0\}|\leq n\}.
\end{split}\]
Instead of $\absamen_p^p(\NN)$, $\absamen_p^p(n,k)$ and $\absamen_p^p(n)$ we write simply $\absamen_p(\NN)$, $\absamen_p(n,k)$ and $\absamen_p(n)$, respectively.
\end{definition}

Let us start with the following two basic observations.

\begin{fact}\label{fact:choiceOfZero}
Let $\MM$ be a pointed $p$-metric space and $\NN\subset \MM$ its subspace. Then $\amen_p(\NN,\MM)$ does not depend on the choice of the base point from $\NN$.
\end{fact}
\begin{proof}
Let $x_1,x_2\in\NN$ be two points and for the purpose of this proof given $i\in\{1,2\}$ we denote by $L_{j,x_i}:\F_{p,x_i}(\NN)\to \F_{p,x_i}(\MM)$ the canonical linearization of the inclusion $j:\NN\to\MM$, where Lipchitz-free spaces are constructed using the base point $x_i$. It is easy to check (and well-known) that mapping $\F_{p,x_1}(\NN)\ni\delta(z)\mapsto \delta(z)-\delta(x_1)\in\F_{p,x_2}(\NN)$ induces a surjective isometry between spaces $\F_{p,x_1}(\NN)$ and $\F_{p,x_2}(\NN)$ (and similarly for $\MM$), so we obtain
\[\begin{split}
\norm{L_{j,x_1}^{-1}} & = \sup\Big\{\frac{\norm{\sum_{x\in F} a_x \delta(x)}_{\F_{p,x_1}(\NN)}}{\norm{\sum_{x\in F} a_x \delta(x)}_{\F_{p,x_1}(\MM)}}\colon F\subset \NN\text{ finite}, a\in \R^F\Big\}\\
& = \sup\Big\{\frac{\norm{\sum_{x\in F} a_x (\delta(x)-\delta(x_1))}_{\F_{p,x_2}(\NN)}}{\norm{\sum_{x\in F} a_x (\delta(x)-\delta(x_1))}_{\F_{p,x_2}(\MM)}}\colon F\subset \NN\text{ finite}, a\in \R^F\Big\}\\
& = \sup\Big\{\frac{\norm{\sum_{x\in F} a_x \delta(x)}_{\F_{p,x_2}(\NN)}}{\norm{\sum_{x\in F} a_x \delta(x)}_{\F_{p,x_2}(\MM)}}\colon F\subset \NN\text{ finite}, a\in \R^F\Big\} = \norm{L_{j,x_2}^{-1}}.
\end{split}\]
Thus, the norm of $L_j^{-1}$ does not depend on the choice of the base point.
\end{proof}

\begin{fact}\label{fact:attainMin}
Let $0<p\leq q\leq 1$ and let $(\MM,d,0)$ be a finite pointed $q$-metric space. Let $0\in \NN\subset \MM$ be its subspace. Then there exists $a\in \R^{\NN\setminus\{0\}}$ and $T\in \TT(\MM)$ such that $T(a)=1=\|\sum_{x\in\NN\setminus\{0\}}a_x\delta(x)\|_{\F_p(\MM)}$ and
\begin{equation}\label{eq:normAttain}
\amen_p(\NN,\MM) = \normb{Big}{\sum_{x\in\NN\setminus\{0\}}a_x\delta(x)}_{\F_p(\NN)}.
\end{equation}
\end{fact}
\begin{proof}
Let $L_j$  be the canonical linearization of the inclusion map $j:\NN\to \MM$. Then
\[
\norm{L_j^{-1}} = \sup\Big\{\normb{Big}{\sum_{x\in\NN\setminus\{0\}}a_x\delta(x)}_{\F_p(\NN)}\colon a\in \R^{\NN\setminus\{0\}}, \sum_{x\in\NN\setminus\{0\}} a_x\delta(x)\in S_{\F_p(\MM)}\Big\}
\]
and since $\F_p(\MM)$ is finite-dimensional, the supremum is attained, so there exists $a\in\R^{\NN\setminus\{0\}}$ with $\|\sum_{x\in\NN\setminus\{0\}}a_x\delta(x)\|_{\F_p(\MM)} = 1$ and $\norm{L_j^{-1}} = \|\sum_{x\in\NN\setminus\{0\}}a_x\delta(x)\|_{\F_p(\NN)}$. Finally, by Theorem~\ref{thm:algorithm} there exists $T\in\TT(\MM)$ satisfying $T(a) = 1$, which finishes the proof.
\end{proof}

In the following we mention a condition under which we may omit a vertex with zero coefficient from a tree considered in the algorithm from Theorem~\ref{thm:algorithm}.

\begin{lemma}\label{lem:removeNullVertex}Let $(\MM,d,0)$ be a finite pointed $p$-metric space for some $p\in(0,1)$,  $a\in\R^{\MM\setminus\{0\}}$ and $z\in\MM\setminus\{0\}$ with $a_z=0$. Let $T = (\MM,E,0)\in\TT(\MM)$ be such that $|\desc_T(z)|\geq 1$. Let us denote by $T(a,z)$ the summands in $T(a)^p$ which contain a distance from $z$, that is,
\[
T(z,a):=\sum_{y\in \desc(z)} \abs{c_T(y,a)}^pd^p(z,y) + \absb{Big}{\sum_{y\in \desc(z)} c_T(y,a)}^pd^p(\pred(z),z).
\]
Assume that
\[
            T(z,a)\geq \sum_{y\in \desc(z)} \abs{c_T(y,a)}^pd^p(\pred(z),y).
\]
Then $T(a)\geq \norm{\sum_{x\in \MM\setminus\{0,z\}}a_x\delta(x)}_{\F_p(\MM\setminus\{z\})}$.
\end{lemma}
\begin{proof}
We replace in our tree $T$ edges $\{(\pred(z),z),(z,y)\colon y\in \desc(z)\}$ by edges $\{(\pred(z)),y)\colon y\in\desc(z)\}$. Formally speaking, we
consider the three $T' = (\MM\setminus\{z\},E',0)\in \TT(\MM\setminus\{z\})$ such that 
\[
E' = \Big(E\setminus \{(\pred(z),z),(z,y)\colon y\in\desc(z)\}\Big)\cup \{(\pred(z),y)\colon y\in \desc(z)\}.
\]
Then for the tree $T'$ we obtain
\[
T(a)^p \geq T'(a)^p \geq \normb{Big}{\sum_{x\in \MM\setminus\{0,z\}}a_x\delta(x)}^p_{\F_p(\MM\setminus\{z\})}.\qedhere
\]
\end{proof}

Our first application of Lemma~\ref{lem:removeNullVertex} is the following.

\begin{proposition}\label{prop:omitPoints}
Let $(\MM,d,0)$ be a finite pointed $p$-metric space for some $p\in(0,1)$ and $z\in \MM\setminus\{0\}$. Then for any $a\in \R^{\MM\setminus\{0\}}$ with $a_z=0$ we have that $\norm{\sum_{x\in \MM\setminus\{0\}} a_x\delta(x)}_{\F_p(\MM)}$ is equal to the minimum of $\norm{\sum_{x\in \MM\setminus\{0,z\}} a_x\delta(x)}_{\F_p(\MM\setminus\{z\})}$ and all the numbers $T(a)$, where $T\in\TT(\MM)$ is such that $|\desc_T(z)|\geq 2$.
\end{proposition}
\begin{proof}
Put $\NN = \MM\setminus\{z\}$. Pick $T = (\MM,E,0)\in \TT(\MM)$. If $z\in\leaf(T)$, then we have $V^T_z = \{z\}$, so $\sum_{y\in V^T_z} a(y) = a(z) = 0$ and therefore if we omit the point $z$ from the tree, the value $T(a)$ will be the same. Formally speaking, for the tree $T' = (\NN,E',0)\in \TT(\NN)$ with $E' = E\setminus \{e^T_z\}$ we obtain $T(a) = T'(a)\geq \norm{\sum_{x\in \NN\setminus\{0\}} a_x\delta(x)}_{\F_p(\NN)}$.

Now consider the case when $|\desc(z)|=1$. We shall prove that in this case we may apply Lemma~\ref{lem:removeNullVertex}. Indeed, in this case there is a unique $y\in \desc(z)$ and we have
\[
\abs{c_T(y,a)}^pd^p(y,\pred(z)) \leq \abs{c_T(y,a)}^p (d^p(y,z) + d^p(z,\pred(z))) = T(z,a),
\]
so applying Lemma~\ref{lem:removeNullVertex} we obtain $
T(a)\geq \normb{Big}{\sum_{x\in \NN\setminus\{0\}} a_x\delta(x)}_{\F_p(\NN)}$.
\end{proof}

The following implies that if $\NN$ is a $p$-metric space and we want to minimalize the norm $\|\sum_{x\in \NN}a_x \delta(x)\|_{\F_p(\MM)}$ over all the possible $p$-metric spaces $\MM$ with $\NN\subset \MM$, it suffices to consider the case when $\MM\setminus\NN$ has at most $|\NN| - 2$ points. Moreover, if $|\MM\setminus \NN| = |\NN| - 2$, we obtain a restriction on the trees we need to consider.

\begin{proposition}\label{prop:addTooManyPoints}
    Let $(\MM,d,0)$ be a pointed $p$-metric space for some $p\in(0,1)$, let $0\in \NN\subset \MM$ be a finite set, $|\NN|\geq 3$ and $a\in \R^{\MM\setminus\{0\}}$ be such that $a|_{\MM\setminus \NN}\equiv 0$.
    \begin{enumerate}[(a)]
        \item\label{it:addTooManyPoints} If $|\MM\setminus \NN| > |\NN| - 2$, then 
        \[\begin{split}
        \normb{Big}{\sum_{x\in \NN\setminus\{0\}} a_x\delta(x)}_{\F_p(\MM)} = \min\Big\{\normb{Big}{\sum_{x\in \NN\setminus\{0\}} a_x\delta(x)}_{\F_p(F)}\colon & \NN\subset F\subset\MM \text{ and } \\& |F\setminus\NN| = |\NN|-2\Big\}.
        \end{split}\]
        In particular, $\absamen_p^q(n) = \absamen_p^q(n,2n-1)$ for every $n\in\N$, $n\geq 2$ and $q\in [p,1]$.
        \item If $|\MM\setminus \NN| = |\NN| - 2$, then 
        \[\begin{split}\normb{Big}{\sum_{x\in \NN\setminus\{0\}} a_x\delta(x)}_{\F_p(\MM)} = \min\Big\{& \min \Big\{\normb{Big}{\sum_{x\in \NN\setminus\{0\}} a_x\delta(x)}_{\F_p(\MM\setminus\{y\})}\colon y\in \MM\setminus \NN\Big\},\\
        & \min\{T(a)\colon T\in\mathfrak{G}\}\Big\},\end{split}\]
        where $\mathfrak{G}\subset \TT(\MM)$ consists of those $T\in \TT(\MM)$ for which $\leaf(T) = \NN\setminus \{0\}$, $|\desc(0)|=1$ and for every $z\in \MM\setminus \NN$ we have $|\desc(z)| = 2$.
    \end{enumerate}
\end{proposition}
\begin{proof}By Fact~\ref{fact:wlogFinite}, we may without loss of generality assume that $\MM$ is finite. Let us denote by $\mathfrak{G}'\subset \TT(\MM)$ those trees $T\in \TT(\MM)$ for which we have $|\desc(z)| \geq 2$ for every $z\in \MM\setminus \NN$. By Proposition~\ref{prop:omitPoints}, we have that $\norm{\sum_{x\in \NN\setminus\{0\}} a_x\delta(x)}_{\F_p(\MM)}$ is equal to the minimum of all $\norm{\sum_{x\in \NN\setminus\{0\}} a_x\delta(x)}_{\F_p(\MM\setminus\{y\})}$, $y\in \MM\setminus\NN$ and all the numbers $T(a)$, $T\in \mathfrak{G}'$.

Note that for every $T\in \mathfrak{G}'$ we have that $\desc(z)$, $z\in \MM\setminus \NN$ are pairwise disjoint sets, each of which is disjoint from $\desc(0)$. Thus,
\begin{equation}\label{eq:cardinalitArgument1}
|\MM\setminus \NN| + |\NN\setminus\{0\}| = |\MM\setminus \{0\}| \geq \sum_{z\in \MM\setminus \NN} |\desc(z)| + |\desc(0)|\geq 2|\MM\setminus \NN| + 1.
\end{equation}
So we obtain that $|\MM\setminus \NN|\leq |\NN\setminus\{0\}| - 1 = |\NN| - 2$.

Thus, if $|\MM\setminus \NN| > |\NN| - 2$ we have $\mathfrak{G}' = \emptyset$ which proves that there exist $y\in \MM\setminus\NN$ such that for $F':=\MM\setminus\{y\}$ we have 
  $\norm{\sum_{x\in \NN\setminus\{0\}} a_x\delta(x)}_{\F_p(\MM)}=\norm{\sum_{x\in \NN\setminus\{0\}} a_x\delta(x)}_{\F_p(F')}$. If $|F'\setminus\NN| = |\NN| - 2$, then \ref{it:addTooManyPoints} is proved; otherwise, we inductively apply the already proven to omit more and more points until we find a finite set $\NN\subset F\subset \MM$ with $|F\setminus \NN| = |\NN|-2$ such that $\norm{\sum_{x\in \NN\setminus\{0\}} a_x\delta(x)}_{\F_p(\MM)}=\norm{\sum_{x\in \NN\setminus\{0\}} a_x\delta(x)}_{\F_p(F)}$. This proves \ref{it:addTooManyPoints}.

If $|\MM\setminus \NN| = |\NN| - 2$, then there are equalities in \eqref{eq:cardinalitArgument1} and so we have $|\desc(0)| = 1$, $|\desc(z)|=2$ for $z\in \MM\setminus \NN$ and $\MM\setminus (\{0\}\cup \desc(0)) = \bigcup_{z\in \MM\setminus \NN} \desc(z)$. Since $|\desc(z)|=2$ for $z\in \MM\setminus \NN$, we have $\leaf(T)\subset \NN\setminus\{0\}$. Finally, let us show that $\NN\setminus\{0\}\subset\leaf(T)$. Indeed, if there was $x\in \NN\setminus\{0\}$ and $y \in \desc(x)$ then $y\in \desc(z)$ for some $z\in \MM\setminus\NN$, so $x = \pred(y) = z$ which is a contradiction.
\end{proof}

\subsection{Isometric \texorpdfstring{$p$-}-amenability for spaces with three points}\label{subsec:isometricAmen}
The main result of this Subsection is Theorem~\ref{thm:isomAmen}. The technical part of the proof is contained in the following Lemma.

\begin{lemma}\label{lem:technical}Let $(\MM,d,0)$ be a finite pointed $p$-metric space for some $p\in(0,1)$ with $\MM = \{0,x,y,z\}$. Pick  $a\in\R^{\MM\setminus\{0\}}$ with $a_z=0$. Consider the tree $T = (\MM,E,0)\in\TT(\MM)$ with $E = \{(0,z),(z,x),(z,y)\}$ and put $d_{x0}:=d(x,0)$, $d_{y0}:=d(y,0)$, $d_{xy}:=d(x,y)$. Put
\[
C:=|a_x+a_y|^p \frac{d_{x0}^p+d_{y0}^p-d_{xy}^p}{2} + |a_x|^p \frac{d_{xy}^p + d_{x0}^p - d_{y0}^p}{2} + |a_y|^p \frac{d_{xy}^p + d_{y0}^p - d_{x0}^p}{2}.
\]
Then we have $T(a)^p\geq C$
and
\begin{equation}\label{eq:technical}
C\leq \norm{a_x\delta(x) + a_y\delta(y)}_{\F_p(\{0,x,y\})}^p.
\end{equation}
Moreover, in \eqref{eq:technical} equality holds if and only if one of the following equalities hold
\[d_{x0}^p + d_{y0}^p = d_{xy}^p, \quad d_{x0}^p + d_{xy}^p = d_{y0}^p,\quad  d_{y0}^p + d_{xy}^p = d_{x0}^p,\quad a_xa_y=0,\quad a_x=-a_y.\]
\end{lemma}
\begin{proof}The first inequality follows from the following computation
\[\begin{split}
& \mkern-36mu \sum_{w\in\MM\setminus\{0\}} \absb{Big}{\sum_{v\in V_{w}}a_v}^pd^p(e_w) = |a_x|^pd^p(e_x) + |a_y|^pd^p(e_y) + |a_x+a_y|^pd^p(e_{z})\\
& =\tfrac{|a_x+a_y|^p + |a_x|^p - |a_y|^p}{2}(d^p(e_{x}) + d^p(e_z)) + \tfrac{|a_x+a_y|^p + |a_y|^p - |a_x|^p}{2}(d^p(e_y) + d^p(e_z))\\ 
& \quad+ \tfrac{|a_x|^p + |a_y|^p - |a_x+a_y|^p}{2}(d^p(e_x) + d^p(e_y))\\
& \geq \tfrac{|a_x+a_y|^p + |a_x|^p - |a_y|^p}{2}d^p(x,0) + \tfrac{|a_x+a_y|^p + |a_y|^p - |a_x|^p}{2}d^p(y,0)\\ 
& \quad+ \tfrac{|a_x|^p + |a_y|^p - |a_x+a_y|^p}{2}d^p(x,y) = C.
\end{split}\]
Now, we easily check that
\[\begin{split}
C - |a_x|^pd_{x0}^p - |a_y|^pd_{y0}^p & = \Big(d_{x0}^p + d_{y0}^p - d_{xy}^p\Big)\Big(\frac{|a_x+a_y|^p - |a_x|^p - |a_y|^p}{2}\Big)\leq 0,\\
C - |a_x+a_y|^pd_{x0}^p - |a_y|^pd_{xy}^p & = \Big(d_{x0}^p + d_{xy}^p - d_{y0}^p\Big)\Big(\frac{|a_x|^p - |a_y|^p - |a_x+a_y|^p}{2}\Big)\leq 0,\\
C - |a_x+a_y|^pd_{y0}^p - |a_x|^pd_{xy}^p & = \Big(d_{y0}^p + d_{xy}^p - d_{x0}^p\Big)\Big(\frac{|a_y|^p - |a_x+a_y|^p - |a_x|^p}{2}\Big)\leq 0,
\end{split}\]
which easily implies that
\[\begin{split}
C & \leq \min\{|a_x|^pd_{x0}^p + |a_y|^pd_{y0}^p, |a_x+a_y|^pd_{x0}^p + |a_y|^pd_{xy}^p, |a_x+a_y|^pd_{y0}^p + |a_x|^pd_{xy}^p\}\\
&  \stackrel{\text{Corollary \ref{cor:twoPointsFormula}}}{=} \norm{a\delta(x)+b\delta(y)}_{\F_p(\NN)}^p,
\end{split}\]
and the first inequality above is actually equality if and only if one of the following equalities hold: $d_{x0}^p + d_{y0}^p = d_{xy}^p$, $d_{x0}^p + d_{xy}^p = d_{y0}^p$, $d_{y0}^p + d_{xy}^p = d_{x0}^p$, $|a_x+a_y|^p = |a_x|^p + |a_y|^p$, $|a_x+a_y|^p + |a_y|^p = |a_x|^p$, $|a_y|^p = |a_x+a_y|^p + |a_x|^p$. Finally, we observe that for $p<1$ one of the last three equalities hold if and only if $a_xa_y=0$ or $a_x=-a_y$.
\end{proof}

The following construction shall provide us with an optimal extension of a $p$-metric space containing $3$ points, which minimizes the Lipschitz-free space $p$-norm.

\begin{proposition}\label{prop:optimalAddPoint}Let $(\MM,d,0)$ be a finite pointed $p$-metric space for some $p\in(0,1)$ with $\MM=\{0,x,y,z\}$ and let $\NN = \{0,x,y\}$. Assume that no one of the $p$-inequalities among distances between points $\{0,x,y\}$ is equality and consider the $p$-metric $d'$ on $\MM$ such that $d' = d$ on $\NN\times \NN$ and
\[
d'(x(1),z) = \Big(\frac{d^p(x(2),x(1)) + d^p(x(3),x(1)) - d^p(x(2),x(3))}{2}\Big)^{1/p}
\]
for any bijection $x:\{1,2,3\}\to \NN$.

Then for $\MM' = (\MM,d')$ and any $a,b\in\R$ we have
\begin{equation}\label{eq:optimalAddingPoint}
    \norm{a\delta(x) + b\delta(y)}_{\F_p(\MM')}\leq \norm{a\delta(x) + b\delta(y)}_{\F_p(\MM)}.
\end{equation}
Moreover if we denote $d_x = d(x,0)$, $d_y = d(y,0)$, $d_{xy}=d(x,y)$, then
\begin{equation}\label{eq:valueOfFreeNorm}
    \norm{a\delta(x) + b\delta(y)}_{\F_p(\MM')}^p = d_x^p \tfrac{|a+b|^p + |a|^p - |b|^p}{2} + d_y^p \tfrac{|a+b|^p + |b|^p - |a|^p}{2} + d_{xy}^p\tfrac{|a|^p + |b|^p - |a+b|^p}{2}
\end{equation}
and $\norm{a\delta(x) + b\delta(y)}_{\F_p(\MM')}^p = \norm{a\delta(x) + b\delta(y)}_{\F_p(\NN)}^p$ if and only if $ab=0$ or $a=-b$.
\end{proposition}
\begin{proof}Using the assumption that $p$-inequalities between points $\{0,x,y\}$ are never equalities, it is easy to check that $(\MM,d')$ is a $p$-metric space. Pick $\alpha\in \R^{\{x,y\}}$ and in order to shorten the notation write $a:=\alpha_x$ and $b:=\alpha_y$. Put
\[C(a,b):=d_x^p \tfrac{|a+b|^p + |a|^p - |b|^p}{2} + d_y^p \tfrac{|a+b|^p + |b|^p - |a|^p}{2} + d_{xy}^p\tfrac{|a|^p + |b|^p - |a+b|^p}{2}.\]
There is only one tree $T = (\MM,E,0)$ such that $|\desc(z)|=2$ and this is the tree with edges $(0,z)$, $(z,x)$, $(z,y)$ and for this $T$ we obtain 
\[
T_d(\alpha)^p \stackrel{\text{Lemma~\ref{lem:technical}}}{\geq} |a+b|^p\tfrac{d_{x}^p + d_y^p - d_{xy}^p}{2} + |a|^p\tfrac{d_x^p + d_{xy}^p - d_{y}^p}{2} + |b|^p\tfrac{d_y^p + d_{xy}^p - d_x^p}{2} = C(a,b),
\]
so by Proposition~\ref{prop:addTooManyPoints} we have that 
\begin{equation}\label{eq:optimalAddinPointReformulated}
\|a\delta(x)+b\delta(y)\|_{\F_p(\MM)}^p\geq \min\{\|a\delta(x)+b\delta(y)\|_{\F_p(\NN)}^p, C(a,b)\}.
\end{equation}
Moreover, we observe that in the $p$-metric space $\MM'$ we have
\begin{equation}\begin{split}\label{eq:optimalExtensionFormula}
T_{d'}(\alpha)^p & = |a|^pd'^p(x,z) + |b|^pd'^p(y,z) + |a+b|^pd'^p(z,0)= C(a,b)\\
&  \stackrel{\text{Lemma~\ref{lem:technical}}}{\leq} \norm{a\delta(x)+b\delta(y)}_{\F_p(\NN)}^p,
\end{split}\end{equation}
so we have $\norm{a\delta(x)+b\delta(y)}_{\F_p(\MM')}^p = C(a,b)$ and in conjuction with \eqref{eq:optimalAddinPointReformulated} we obtain \eqref{eq:optimalAddingPoint}. Finally, by Lemma~\ref{lem:technical} we have equality in \eqref{eq:optimalExtensionFormula} if and only if $ab=0$ or $a=-b$, which proves that  $\norm{a\delta(x) + b\delta(y)}_{\F_p(\MM')}^p = \norm{a\delta(x) + b\delta(y)}_{\F_p(\NN)}^p$ if and only if $ab=0$ or $a=-b$.
\end{proof}

\begin{theorem}\label{thm:isomAmen}Let $(\NN,d,0)$ be a finite pointed $p$-metric space for some $p\in(0,1)$ with $\NN=\{0,x,y\}$. Then the following conditions are equivalent.
\begin{enumerate}[(a)]
    \item\label{it:2pointsCondition} One of the $p$-inequalities among distances between points $\{0,x,y\}$ is equality.
    \item\label{it:2pointsAmenable} We have $\absamen_p(\NN) = 1$. 
\end{enumerate}
Moreover, if \ref{it:2pointsAmenable} does not hold, then $\absamen_p(\NN) = \amen_p(\NN,\MM')$, where $\MM' = \NN\cup\{z\}$ and the $p$-metric $d'$ on $\MM'$ is given by the formula from Proposition~\ref{prop:optimalAddPoint}.
\end{theorem}
\begin{proof}
\ref{it:2pointsCondition}$\Rightarrow$\ref{it:2pointsAmenable}: By Proposition~\ref{prop:addTooManyPoints}, it suffices to prove that $\amen_p(\NN,\MM) = 1$ whenever $\MM = \NN\cup \{z\}$ for some $z\notin \NN$. Pick $a\in\R^{\NN\setminus\{0\}}$. There is only one tree $T = (\MM,E,0)$ with $|\desc(z)|=2$ and this is the tree with edges $(0,z)$, $(z,x)$, $(z,y)$. But for this tree we obtain from Lemma~\ref{lem:technical} that $T(a)^p\geq C$, where 
\[
C = \norm{a_x\delta(x) + a_y\delta(y)}_{\F_p(\NN)}^p.
\]
Thus, by Proposition~\ref{prop:addTooManyPoints} we obtain $\norm{a_x\delta(x) + a_y\delta(y)}_{\F_p(\NN)} = \norm{a_x\delta(x) + a_y\delta(y)}_{\F_p(\MM)}$ and since $a$ was arbitrary, $\NN$ is isometrically $p$-amenable in $\MM$.\\
\ref{it:2pointsAmenable}$\Rightarrow$\ref{it:2pointsCondition}: If no one of the $p$-inequalities among distances between points $\{0,x,y\}$ is equality, we consider the $p$-metric space $\MM' = \{0,x,y,z\}$ with the distance $d'$ as in Proposition~\ref{prop:optimalAddPoint}. Then by Proposition~\ref{prop:optimalAddPoint} we have $\|a\delta(x) + b\delta(y)\|_{\F_p(\MM')} < \|a\delta(x) + b\delta(y)\|_{\F_p(\NN)}$ for example for the choice $a=b=1$, so we have $\absamen_p(\NN) > 1$ and the ``Moreover'' part follows from Proposition~\ref{prop:optimalAddPoint} as well.
\end{proof}

\subsection{Reduction to special \texorpdfstring{$p$}{p}-metric spaces and their subspaces}\label{subsec:reductionToTrees} The aim of this subsection is to prove Theorem~\ref{thm:reductionExamples} which substantially reduces the class of $p$-metric spaces we need to consider when dealing with the estimate of the value of $\absamen_p^q(n,k)$. We start with the description of $p$-metric spaces given by weighted trees.

\begin{definition}Let $T = (V,E,w)$ be a finite weighted tree, that is, $T=(V,E)$ is a tree (unoriented) and $w:E\to (0,\infty)$ is a function. Given $x,y\in \MM$ by the \emph{path} from $x$ to $y$ in $T$ we understand the sequence of pairwise distinct vertices $x_1,\ldots,x_n\in V$ with $x_1=x$, $x_n=y$ satisfying that $\{x_i,x_{i+1}\}\in E$ for every $i=1,\ldots,n-1$. Note that since $T$ is a tree, such a path is always unique. Further, for $p\in (0,1]$ we define the $p$-metric $d_{w,p}$ on $V$ by the formula
\[
d_{w,p}(x,y):=\Big(\sum_{i=1}^{n-1}w^p(x_i,x_{i+1})\Big)^{1/p},\quad x,y\in V,
\]
where $x_1,\ldots,x_n$ is the path from $x$ to $y$ in $T$. We say that the $p$-metric space $V_{T,p}:=(V,d_{w,p})$ is the \emph{$p$-metric space associated to the weighted tree $T$}.
\end{definition}

\begin{notation}
Let $(\MM,d,0)$ be a finite pointed $p$-metric space for some $p\in(0,1]$ and $T\in\T(\MM)$. Endow the tree $T$ with weights $w(e):=d(e)$, $e\in E$. By $d_{T,p}$ we denote the $p$-metric $d_{w,p}$ on the $p$-metric space $\MM_{T,p} = (\MM,d_{w,p})$ associated to the weighted tree $T$. If $\NN\subset \MM$ we put $\NN_{T,p}:=(\NN,d_{T,p})$. When $p$ is clear from the context, we write $\MM_T$, $d_T$ and $\NN_T$ instead of $\MM_{T,p}$, $d_{T,p}$ and $\NN_{T,p}$, respectively.\\
(Warning: if $\NN\subsetneq \MM$, then $T\notin \TT(\NN)$, so the symbol $\NN_T$ denotes different $p$-metric spaces depending on whether $T\in \T(\NN)$ or $T\in\TT(\MM)$ for some $\MM\supset \NN$.)
\end{notation}

The starting point of our argument is to observe that when trying to find the value of $\absamen_p^q(n,k)$, it suffices to consider $p$-metric spaces coming from weighted trees and their subspaces.

\begin{lemma}\label{lem:modifySubspace}Let $0<p\leq q\leq 1$ and let $(\MM,d,0)$ be a finite pointed $q$-metric space. Let $0\in \NN\subset \MM$ be its subspace, $a\in \R^{\NN\setminus\{0\}}$ and $T\in \TT(\MM)$.

Then for the $q$-metric space $\MM_{T}$ and its subspace $\NN_{T}$ we have $T_{d_{T}}(a) = T_d(a)$ and 
\begin{equation}\label{eq:improveSmall}
\Big\|\sum_{x\in\NN\setminus\{0\}} a_x\delta(x)\Big\|_{\F_p(\NN_{T})}\geq \Big\|\sum_{x\in\NN\setminus\{0\}} a_x\delta(x)\Big\|_{\F_p(\NN)}.
\end{equation}

In particular, for every $a\in \R^{\NN\setminus\{0\}}$ and $T\in \TT(\MM)$ satisfying $T(a)=1=\|\sum_{x\in\NN\setminus\{0\}}a_x\delta(x)\|_{\F_p(\MM)}$ and $\amen_p(\NN,\MM) = \|\sum_{x\in\NN\setminus\{0\}}a_x\delta(x)\|_{\F_p(\NN)}$, we have 
\[
\amen_p(\NN,\MM)\leq \frac{\|\sum_{x\in\NN\setminus\{0\}}a_x\delta(x)\|_{\F_p(\NN_T)}}{T_{d_T}(a)}\leq \amen_p(\NN_{T},\MM_{T}).
\]
\end{lemma}
\begin{proof}Since for any two points $x,y\in\MM$ connected by an edge in $T$ we have $d_T(x,y) = d(x,y)$, by the definition we immediately obtain $T_{d_T}(a) = T_{d}(a)$. Moreover, using the triangle $q$-inequality we obtain $d_T(x,y)\geq d(x,y)$ for any $x,y\in\MM$. Thus, for any $S\in\TT(\NN)$ we obtain $S_{d}(a)\leq S_{d_T}(a)$, which by Theorem~\ref{thm:algorithm} implies \eqref{eq:improveSmall}.

In order to prove the ``In particular'' part, let $a\in \R^{\NN\setminus\{0\}}$ and $T\in \TT(\MM)$ be such that $T(a)=1=\|\sum_{x\in\NN\setminus\{0\}}a_x\delta(x)\|_{\F_p(\MM)}$ and $\amen_p(\NN,\MM) = \|\sum_{x\in\NN\setminus\{0\}}a_x\delta(x)\|_{\F_p(\NN)}$.  Then by the already proven part we have $T_{d_{T}}(a) = T_{d}(a)=1$ and so, denoting by $L_{j,T}$ the canonical linearization of the inclusion map, $j_T:\NN_T\to \MM_T$, 
\[\begin{split}
\amen_p(\NN_{T},\MM_{T}) & = \|L_{j,T}^{-1}\| \geq \frac{\|\sum_{x\in\NN\setminus\{0\}}a_x\delta(x)\|_{\F_p(\NN_T)}}{\|\sum_{x\in\NN\setminus\{0\}}a_x\delta(x)\|_{\F_p(\MM_T)}}\geq  \frac{\|\sum_{x\in\NN\setminus\{0\}}a_x\delta(x)\|_{\F_p(\NN_T)}}{T_{d_T}(a)}\\
& \stackrel{\eqref{eq:improveSmall}}{\geq} \frac{\|\sum_{x\in\NN\setminus\{0\}}a_x\delta(x)\|_{\F_p(\NN)}}{T_{d}(a)} = \amen_p(\NN,\MM).
\end{split}\]
\end{proof}

Now, we shall observe that when trying to find the value of $\absamen_p^q(n,k)$, it suffices to consider $p$-metric spaces coming from weighted trees and their subspaces given by leaves of the corresponding tree. We start with the following technical observation.

\begin{lemma}\label{lem:splitTheTree}Let $p\in (0,1]$, $(\MM,d,0)$ be a finite pointed $p$-metric space and $M_1, M_2\subset \MM$ subsets satisfying that $M_1\cup M_2 = \MM$, $0\in M_2$, $M_1\cap M_2 = \{x_0\}$ for some $x_0\in \MM\setminus \{0\}$ and $|M_1|\geq 2$. Consider now the pointed $p$-metric spaces $\MM_1 = (M_1,d,x_0)$ and $\MM_2 = (M_2,d,0)$. Assume that $T\in\TT(\MM)$, $T_1\in \TT(\MM_1)$, $T_2\in\TT(\MM_2)$ are such that $E^T = E^{T_1}\cup E^{T_2}$ (that is, if we ``glue the trees $T_1$ and $T_2$ at the vertex $x_0$'', we obtain the tree $T$). Let $a\in \R^{\MM\setminus\{0\}}$ be given. Then
\[
T(a)^p = T_1(a_1)^p + T_2(a_2)^p,
\]
where $a_1:=a|_{\MM_1\setminus\{x_0\}}$ and 
\[
a_2(y):=\begin{cases} a(y) & \text{ if }y\in\MM_2\setminus\{0,x_0\},\\
\sum_{z\in \MM_1} a(z) & \text{ if }y = x_0.
\end{cases}
\]
\end{lemma}
\begin{proof}By the assumption that $|M_1|\geq 2$, both $\MM_1$ and $\MM_2$ are nontrivial pointed metric spaces (that is, they contain at least two points).
Note that for every $x\in \MM_1\setminus\{x_0\}$ we have $V_x^T = V_x^{T_1}$ and so $c_T(x,a) = c_{T_1}(x,a_1)$. For for $x\in \MM_2\setminus \{0\}$ we \emph{claim} that $c_T(x,a) = c_{T_2}(x,a_2)$.

Indeed, pick $x\in \MM_2\setminus \{0\}$. Then $V_x^T = V_x^{T_2}$ if $x_0\notin V_{x}^T$ and otherwise $V_x^T = \MM_1\cup V_x^{T_2}$. Thus, $c_T(x,a) = c_{T_2}(x,a_2)$ if $x_0\notin V_{x}^T$ and otherwise we have
\[\begin{split}
c_T(x,a) & = \sum_{y\in \MM_1\cup V_x^{T_2}} a(y) = \sum_{y\in V_x^{T_2}\setminus \{x_0\}} a (y) + \sum_{y\in \MM_1} a(y)\\
& = \sum_{y\in V_x^{T_2}\setminus \{x_0\}} a_2(y)  + a_2(x_0) = c_{T_2}(x,a_2),
\end{split}\]
which proves the claim.

Moreover, we have $e_x^T = e_x^{T_1}$ and $e_x^T = e_x^{T_2}$ for $x\in \MM_1\setminus\{x_0\}$ and $x\in \MM_2\setminus\{0\}$, respectively. Thus, we obtain
\[\begin{split}
T(a)^p & = \sum_{x\in\MM\setminus\{0\}} |c_T(x,a) d(e^T_x)|^p = \sum_{x\in\MM_1\setminus\{x_0\}} |c_{T}(x,a) d(e^T_x)|^p + \sum_{x\in\MM_2\setminus\{0\}} |c_{T}(x,a)d(e^T_x)|^p\\ 
&  =  \sum_{x\in\MM_1\setminus\{x_0\}} |c_{T_1}(x,a_1)d(e^{T_1}_x)|^p + \sum_{x\in\MM_2\setminus\{0\}} |c_{T_2}(x,a_2)d(e^{T_2}_x)|^p \\
& = T_1(a_1)^p + T_2(a_2)^p.
\end{split}\]
\end{proof}

\begin{proposition}\label{prop:leafesInSubspace}
Let $p\in (0,1)$, $\MM$ be a finite $p$-metric space,  $\NN\subset \MM$ its subspace such that $\amen_p(\NN,\MM) > \amen_p(\NN,\MM\setminus\{z\})$ for every $z\in\MM\setminus\NN$. Pick $a\in \R^{\NN\setminus\{0\}}$ and $T\in\TT(\MM)$ with 
\[
\amen_p(\NN,\MM) = \frac{\|\sum_{x\in\NN\setminus\{0\}}a_x\delta(x)\|_{\F_p(\NN)}}{T(a)}.
\]
If $\leaf(T)\neq \NN\setminus\{0\}$, then there exists $\MM'\subsetneq \MM$ and $\NN'\subset \MM'$, $|\NN'|<|\NN|$ such that $\amen_p(\NN,\MM)\leq \amen_p(\NN',\MM')$.
\end{proposition}
\begin{proof}
Since $\amen_p(\NN,\MM) > \amen_p(\NN,\MM\setminus\{z\})$ for every $z\in\MM\setminus\NN$, by Proposition~\ref{prop:omitPoints} we have $|\desc_T(z)|\geq 2$ for every $z\in\MM\setminus\NN$ and so in particular $\leaf(T)\subset \NN\setminus\{0\}$. Since $\leaf(T)\neq \NN\setminus\{0\}$, there exists $x_0\in \NN\setminus (\leaf(T)\cup \{0\})$. Put $M_1:=V_{x_0}^T$, $M_2:=(\MM\setminus M_1)\cup \{x_0\}$ and consider the $p$-metric spaces $\MM_1:=(M_1,d,x_0)$ and $\MM_2:=(M_2,d,0)$. For $i\in\{1,2\}$, let $T_i\in\TT(\MM_i)$ denote restrictions of the tree $T$ to vertices $\MM_i$ (that is, $T_i = (\MM_i,E^T\cap (\MM_i)^2,0_{\MM_i})$). By Lemma~\ref{lem:splitTheTree} (the assumption that $|M_1|\geq 2$ is satisfied, because $x_0\notin \leaf(T)$ and so $|V_{x_0}^T|\geq 2$) for $a_1:=a|_{\MM_1\setminus\{x_0\}}$ and
\[
a_2(y):=\begin{cases} a(y) & \text{ if }y\in\MM_2\setminus\{0,x_0\},\\
\sum_{z\in \MM_1} a(z) & \text{ if }y = x_0
\end{cases}
\]
we obtain $T(a)^p = T_1(a_1)^p + T_2(a_2)^p$. Further, put $\NN_1:=(M_1\cap\NN,d,x_0)$, $\NN_2:=(M_2\cap \NN,d,0)$ and, using Theorem~\ref{thm:algorithm}, for $i\in\{1,2\}$ pick $S_i\in\TT(\NN_i)$ with $S_i(a_i|_{\NN_i}) = \|\sum_{x\in \NN_i} a_i(x) \delta(x)\|_{\F_p(\NN_i)}$. Let us denote by $S\in\TT(\NN)$ the union of trees $S_1$ and $S_2$, that is, $S = (\NN,E^{S_1}\cup E^{S_2},0)$. Note that since $\leaf(T)\subset \NN$, we have $|\NN_1|\geq 2$ and that $\{x\in\MM\colon a(x)\neq 0\}\subset \NN\cup \{x_0\}=\NN$. Thus, again by Lemma~\ref{lem:splitTheTree}, we have $S(a|_{\NN})^p = S_1(a_1|_{\NN_1})^p + S_2(a_2|_{\NN_2})^p$. Finally, we obtain
\[\begin{split}
\amen_p(\NN, & \MM)^p = \frac{\|\sum_{x\in\NN\setminus\{0\}}a_x\delta(x)\|^p_{\F_p(\NN)}}{T(a)^p} \leq \frac{S(a|_{\NN})^p}{T(a)^p}\\
&  = \frac{S_1(a_1|_{\NN_1})^p + S_2(a_2|_{\NN_2})^p}{T_1(a_1)^p + T_2(a_2)^p}
 \leq \max\Big\{\frac{S_1(a_1|_{\NN_1})^p}{T_1(a_1)^p},\frac{S_2(a_2|_{\NN_2})^p}{T_2(a_2)^p}\Big\}\\
& \leq \max\Big\{\frac{\|\sum_{x\in \NN_1} a_1(x)\delta(x)\|^p_{\F_p(\NN_1)}}{\|\sum_{x\in \NN_1} a_1(x)\delta(x)\|^p_{\F_p(\MM_1)}},\frac{\|\sum_{x\in \NN_2} a_2(x)\delta(x)\|^p_{\F_p(\NN_2)}}{\|\sum_{x\in \NN_2} a_2(x)\delta(x)\|^p_{\F_p(\MM_2)}}\Big\}\\
& \leq \max\{\amen_p(\NN_1,\MM_1),\amen_p(\NN_2,\MM_2)\}^p
\end{split}\]
 and so it suffices to put $\NN'=\NN_i$ and $\MM'=\MM_i$ for some $i\in\{1,2\}$. 
\end{proof}

The final reduction is now contained in the following result.

\begin{definition}
For each $n\in\N$, we denote by $\TT(n)$ the set of all the weighted rooted trees with $n+1$ vertices.
\end{definition}

\begin{theorem}\label{thm:reductionExamples}For every $0<p\leq q \leq 1$ with $p<1$ and $n,k\in\N$, $2\leq n<k$ we have
\[\begin{split}
\absamen_p^q(n,k) & = \sup\{\absamen_p^q(n,k-1),\; \amen_p(\leaf(T)\cup \{0\},\MM_{T,q})\colon T\in \TT(k),\; |\leaf(T)| = n\}\\
& = \sup\Big\{\absamen_p^q(n,k-1),\\
& \qquad \sup_{a\in \R^{\leaf(T)}}\frac{\|\sum_{x\in\leaf(T)} a_x\delta(x)\|_{\F_p(\leaf(T)\cup\{0\})}}{T_{d_{T,q}}(a)}\colon T\in \TT(k),\; |\leaf(T)| = n\Big\}.
\end{split}\]

In particular, for every $n\in\N$, $n\geq 2$ we have
\[\begin{split}
\absamen_p^q(n)= \sup\Big\{\amen_p\big(\leaf(T)\cup \{0\},\MM_{T,q}\big)\colon & T\in \TT(k) \text{ for some }n \leq k\leq 2n-1\\
& \text{ and }|\leaf(T)|=n\Big\}.
\end{split}\]
\end{theorem}
\begin{proof}Assume that $\NN\subset \MM$ are finite $q$-metric spaces with $|\NN|=n+1$, $|\MM| = k+1$ and $\absamen_p(\NN,\MM) > \absamen_p^q(n,k-1)$. By Fact~\ref{fact:attainMin}, there are $a\in \R^{\NN\setminus\{0\}}$ and $T\in \TT(\MM)$ such that $T(a)=1=\|\sum_{x\in\NN\setminus\{0\}}a_x\delta(x)\|_{\F_p(\MM)}$ and
$\amen_p(\NN,\MM) = \norm{\sum_{x\in\NN\setminus\{0\}}a_x\delta(x)}_{\F_p(\NN)}$. Using Proposition~\ref{prop:leafesInSubspace}, we have $\NN = \leaf(T)\cup \{0\}$. Endow now the tree $T$ with weights $w(e):=d(e)$, $e\in E$ (then $T\in\TT(k)$). By Lemma~\ref{lem:modifySubspace} we have 
\[\begin{split}
\amen_p(\NN,\MM) & \leq \frac{\|\sum_{x\in\NN\setminus\{0\}} a_x\delta(x)\|_{\F_p(\NN_{T,q})}}{T_{d_{T,q}}(a)} \leq \sup_{b\in \R^{\leaf(T)}}\frac{\|\sum_{x\in\leaf(T)} b_x\delta(x)\|_{\F_p(\NN_{T,q})}}{T_{d_{T,q}}(b)}\\
& \leq \amen_p(\NN_{T,q},\MM_{T,q}) = \amen_p(\leaf(T)\cup \{0\},\MM_{T,q}).
\end{split}\]
This proves the required formula for $\absamen_p^q(n,k)$.

The ``In particular'' part easily follows using Proposition~\ref{prop:addTooManyPoints}. 
\end{proof}

\subsection{Estimates of values \texorpdfstring{$\absamen_p^q(n,k)$}{aapq(n,k)}} \label{subsec:estimates}
First of all, recall that if $\NN$ is a $C$-Lipschitz retract of $\MM$, then $\amen_p(\NN,\MM)\leq C$, see e.g. \cite[comment below Definition 2.6]{AACD21}. The following short argument gives us our first estimate.

\begin{fact}\label{fact:upperEstimateOneMorPoint}
    Let $0<p\leq q\leq 1$, $\MM$ be a finite $q$-metric space. Then any $\NN\subset \MM$ with $|\MM\setminus\NN|=1$ is $2^{1/q}$-Lipschitz retract of $\MM$.

    In particular, we have $\absamen_p^q(n,n+1)\leq 2^{1/q}$ for any $n\in\Nat$, $n\geq 2$.
\end{fact}
\begin{proof}If $\MM\setminus \NN = \{z\}$, the retraction $r:\MM\to \NN$ is given as $r|_{\NN} = \operatorname{id}$ and $r(z):=a$, where $a\in \NN$ is such that $d(z,\NN) = d(z,a)$. Then
\[
d^q(r(z),r(x)) = d^q(a,x)\leq d^q(a,z) + d^q(z,x)\leq 2d^q(z,x),\quad x\in\NN,
\]
which shows that $r$ is indeed $2^{1/q}$-Lipschitz mapping. This implies that $\amen_p(\NN,\MM)\leq 2^{1/q}$. Since $\NN$ and $\MM$ were arbitrary $q$-metric spaces with $|\MM\setminus\NN|=1$, the ``In particular'' part follows.
\end{proof}

This simple observation already shows that the example from \cite[Theorem 6.1]{AACD} is optimal when adding just one point to our $p$-metric space. The following is based on the example from \cite[Theorem 6.1]{AACD}, but we use rather the more elegant way of a proof, which is using our Theorem~\ref{thm:positiveCoef} instead of the technical \cite[Proposition 4.16]{AACD}.

\begin{proposition}\label{prop:estimateOneMorePoint}
For every $0<p\leq q \leq 1$ with $p<1$ and $n\in\N$, $2\leq n$ we have 
\[
\frac{2^{1/q}}{(1 + n^{p-1})^{1/p}}\leq \absamen_p^q(n,n+1)\leq 2^{1/q}.
\]
In particular, $\sup_{n\in\N}\absamen_p^q(n,n+1) = 2^{1/q}$.
\end{proposition}
\begin{proof}The upper estimate is Fact~\ref{fact:upperEstimateOneMorPoint}.

For the lower estimate pick the weighted tree $T\in\TT(n+2)$ with vertices $V = \{0,z,x_1\ldots,x_{n}\}$, edges $E = \{(0,z), (z,x_i)\colon i=1,\ldots,n\}$ and weights $w(e)=1$ for $e\in E$. Consider $\NN = \{0,x_1,\ldots,x_n\} = \leaf(T)\cup \{0\}\subset \MM_{T,q}$ and $a = \R^{n}$ given as $a(i):=n^{-1/p}$, $i=1,\ldots,n$. By Theorem~\ref{thm:positiveCoef}, we have 
\[
\normb{Big}{\sum_{i=1}^n a_i \delta(x_i)}^p_{\F_p(\NN)} = \sum_{i=1}^n a_i^p d^p(x_i,0) = \sum_{i=1}^n \frac{1}{n}2^{p/q} = 2^{p/q},
\]
while 
\[
T_{d_{T,q}}(a)^p = \Big(\sum_{i=1}^n a_i\Big)^pd^p_{w,q}(0,z) + \sum_{i=1}^n a_i^p d^p_{w,q}(z,x_i) = n^{p-1} + 1.
\]
Thus, by Theorem~\ref{thm:reductionExamples} we have
\[
\absamen_p^q(n,n+1)\geq \frac{\|\sum_{i=1}^n a_i \delta(x_i)\|_{\F_p(\NN)}}{T_{d_{T,q}}(a)}  = \frac{2^{1/q}}{(1 + n^{p-1})^{1/p}},
\]
which proves the lower estimate as well.
\end{proof}
    
A careful inspection and obvious modification of the proof of \cite[Theorem 1.1]{Basso18} gives us also the following generalization of Fact~\ref{fact:upperEstimateOneMorPoint}.

\begin{theorem}\label{thm:basso}
    Let $0<p\leq q\leq 1$, $\MM$ be a finite $q$-metric space and $\NN\subsetneq \MM$ with $|\NN|\geq 2$. Then $\NN$ is $(|\MM\setminus\NN|+1)^{1/q}$-Lipschitz retract of $\MM$.

    In particular, given $n,k\in\Nat$ with $n\geq 2$ and $n<k$ we have
    \begin{itemize}
        \item $\absamen_p^q(n,k)\leq (k-n+1)^{1/q}$,
        \item $\absamen_p^q(n)\leq n^{1/q}$.
    \end{itemize}
\end{theorem}
\begin{proof}The proof of Theorem 1.1 in \cite{Basso18} gives the result for the case of metric spaces, that is, $q=1$ (the only modification of \cite[proof of Theorem 1.1]{Basso18} is that since $\MM$ is finite, we may put $\varepsilon=0$ in \cite[formula (3.1)]{Basso18} and then in the proof it is checked that there is a $(|\MM\setminus\NN|+1)$-Lipschitz retraction $R_\varepsilon:\MM\to\NN$).

For $q<1$ we use that the already proven case to find a $(|\MM\setminus\NN|+1)$-Lipschitz retraction $r:(\MM,d^q)\to (\NN,d^q)$, which is then a $(|\MM\setminus\NN|+1)^{1/q}$-Lipschitz when considered as a mapping from $(\MM,d)$ onto $(\NN,d)$.

The ``In particular'' part immediately follows using that by Proposition~\ref{prop:addTooManyPoints} we have $\absamen_p^q(n) = \absamen_p^q(n,2n-1)$.
\end{proof}

Note that Theorem~\ref{thm:basso} together with Proposition~\ref{prop:addTooManyPoints} implies that if we want to find sequences of $p$-metric spaces $(\NN_n)_{n\in\Nat}$, $(\MM_n)_{n\in\Nat}$ with $\NN_n\subset\MM_n$, $n\in\Nat$ and $\amen_p^q(\NN_n,\MM_n)\to\infty$, we must have $|\NN_n|\to\infty$ and at the same time we may consider only the case when $|\MM_n\setminus\NN_n|\leq |\NN_n| - 2$. That is, if we want to find an example when $\amen_p^q(\NN,\MM)$ is big, we have a lower bound on $|\NN|$ and an upper bound on $|\MM\setminus\NN|$.

On the other hand, even after numerous numerical simulations we were not able to find a single example when $\amen_p^q(\NN,\MM) > 2^{1/q}$. For the case of $q=1$ we are even able to find an upper bound, see Theorem~\ref{thm:metricSolution}. This is based on two ingredients, the first being our Theorem~\ref{thm:reductionExamples} and the second the following result from \cite{Bima}, which is an analogy for $p$-Banach spaces of a result by Matou\v{s}ek \cite{Matousek} who proved this result for the case of $p=1$.

\begin{theorem}[B\'ima]\label{thm:bima}
    Let $T$ be a metric tree, $S\subset T$, $p\in (0,1)$ and $X$ a $p$-Banach space. Then for any Lipschitz map $f:S\to X$ there exists a Lipschitz map $F:T\to X$ extending $f$ such that $\Lip(F)\leq C(p)\Lip(f)$, where $C(p)$ is a constant depending only on $p$.
\end{theorem}

\begin{theorem}\label{thm:metricSolution}
    Let $p\in (0,1)$. Then $\sup_{n\in\Nat}\absamen_p^1(n) < \infty$.

    In particular, given metric spaces $\NN\subset\MM$, $\NN$ is $p$-amenable in $\MM$ with a constant $C(p)$, depending only on $p$ (more precisely, we have $C(p)\leq 7\cdot 12^{1/p-1}$).
\end{theorem}
\begin{proof}Let $C(p)$ be the constant from Theorem~\ref{thm:bima}. By Theorem~\ref{thm:reductionExamples}, it suffices to prove that for every $n\in\Nat$ and $T\in\TT(n)$ we have $\amen_p(\{0\}\cup \leaf(T),\MM_{T,1})\leq C(p)$. But every $\MM_{T,1}$ is naturally subset of a metric tree $T'$, where each edge is replaced by a set isometric to the interval $[0,w(e)]$, so by Theorem~\ref{thm:bima} we may extend the isometry $\delta:\{0\}\cup \leaf(T)\to \F_p(\{0\}\cup \leaf(T))$ to a $C(p)$-Lipschitz map $r:T'\to \F_p(\{0\}\cup \leaf(T))$, in particular by \cite[Lemma 2.7]{AACD21} the mapping $r|_{\MM_{T,1}}$ witnesses that $\{0\}\cup \leaf(T)$ is $p$-complementably amenable in $\MM_{T,1}$ with constant $C(p)$ and therefore we have $\amen_p(\{0\}\cup \leaf(T),\MM_{T,1})\leq C(p)$. Thus, we have $\sup_{n\in\Nat}\absamen_p^1(n)\leq C(p)$ and the ``In particular'' part follows easily from Fact~\ref{fact:wlogFinite}.

Finally, the quantitative estimate on $C(p)$ follows from a careful inspection of the proof of Theorem~\ref{thm:bima}.
\end{proof}

\begin{remark}
    We note that the method of the proof of Theorem~\ref{thm:metricSolution} seems not to be applicable if $\NN\subset\MM$ are $p$-metric spaces as in this case an analogy of Theorem~\ref{thm:bima} is not true. Indeed, once we consider e.g. the $p$-metric analogue of a metric tree given as $T = ([0,1],\abs{\cdot}^{1/p})$, by \cite[Theorem 4.13]{AACD} we have $\F_p(T)\equiv L_p([0,1])$ isometrically and therefore $\F_p(T)^* = \{0\}$ which implies that there are no complemented subspaces of $\F_p(T)$ and therefore it cannot be true that $S\subset T$ with $|S|\geq 2$ is complementably $p$-amenable in $T$.
\end{remark}

As noted already in \cite{AACD20}, Theorem~\ref{thm:metricSolution} implies Proposition~\ref{prop:nets} below, which answers in positive \cite[Question 6.4]{AACD20}. Recall that: a $p$-Banach space $Y$ is said to be \emph{crudely finitely representable} in a $p$-Banach space $X$ supposing there exists $\lambda\geq 1$ such that every finite-dimensional $F\subset Y$ embeds $\lambda$-isomorphically into $X$; given a Banach space $X$ and $b>0$, $N\subset X$ is \emph{$b$-dense in $X$} if for every $x\in X$ we have $d(N,x)\leq b$.

\begin{proposition}\label{prop:nets}
Suppose $N$ is a $b$-dense subset of an infinite-dimensional Banach space $X$ for some $b>0$ and let
$p\in (0,1]$. Then $\F_p(N)$ is crudely finitely representable in $\F_p(X)$ and $\F_p(X)$ is crudely finitely
representable in $\F_p(N)$.
\end{proposition}
\begin{proof}Since $N\subset X$, using Theorem~\ref{thm:metricSolution} we have that $\F_p(N)$ isomorphically embeds into $\F_p(X)$ so trivially $\F_p(N)$ is crudely finitely representable in $\F_p(X)$.

On the other hand, pick a finite-dimensional $F\subset \F_p(X)$. Since $\bigcup_{k\in\Nat}\tfrac{1}{k}N$ is a dense subset of $X$, we have $\F_p(X) = \closedSpan{\delta_X(\tfrac{x}{k})\colon x\in N,\, k\in\Nat}$ and therefore we may without loss of generality assume that $F=\Span\{\delta_X(x_1),\ldots,\delta_X(x_m)\}$ for some $x_1,\ldots,x_m\in \frac{1}{k}N$ and $k\in \Nat$, because any finite dimensional space $2$-isomorphically embeds into space of this form (indeed, given basis $\{b_1,\ldots,b_n\}$ of $F$ and $\varepsilon>0$ small enough we pick $k\in\Nat$ and $x_1,\ldots,x_m\in \tfrac{1}{k}N$ such that $\|b_i-b_i'\|<\varepsilon$ for some $b_i'\in \Span\{\delta_X(x_1),\ldots,\delta_X(x_m)\}$ and then it is easy to check that the mapping $b_i\mapsto b_i'$ extends to a $2$-isomorphic embedding supposing $\varepsilon>0$ was small enough, see e.g. \cite[Lemma 3.3(i)]{CDDK} for the proof). By Theorem~\ref{thm:metricSolution} we have that $F$ as above is $C(p)$-isomorphic to $\Span\{\delta_{\tfrac{1}{k}N}(x_1),\ldots,\delta_{\tfrac{1}{k}N}(x_m)\}$, which in turn is isometric to the space $\Span\{\delta_{N}(kx_1),\ldots,\delta_{N}(kx_m)\}\subset \F_p(N)$ (via linearization of the mapping $\tfrac{1}{k}N\ni x\mapsto kx_i\in N$). Thus, any finite-dimensional subspace $F\subset \F_p(X)$ embeds $2C(p)$-isomorphically into $\F_p(N)$, which finishes the proof.
\end{proof}

Apart from dealing with the issue of finding whether we have $\sup_n\absamen_p(n)<\infty$, one could try to estimate $\absamen_p(n)$ for some small values of $n$.

For $n=2$ the best lower estimate we are able to get is the following. Our numerical simulations suggest that this could be the optimal bound.

\begin{proposition}\label{prop:estimateTwoPoints}
For any $p\in (0,1)$ and $q\in [p,1]$ we have
\[
\absamen_p^q(2)\geq \Big(\frac{2\cdot 2^{p/q}}{2^p(2 - (2-2^p)^{q/p})^{p/q} + 2\cdot(2-2^p)}\Big)^{1/p}.
\]
In particular, $\absamen_p(2) \geq \frac{4}{4+ 2^p(2^p-2)}$ and if $p\in (0,1)$ is such that
\[
\frac{2^{p+1}}{4 + 2^p\big((2 - (2-2^p)^{1/p})^p  - 2\big)} > 1,
\]
then $\absamen_p^1(2) > 1$ (and this is the case e.g. for $p=2/3$).
\end{proposition}
\begin{proof}Pick the weighted tree $T\in\TT(3)$ with vertices $\{0,z,x,y\}$, edges $E=\{(0,z),(z,x),(z,y)\}$ and weights $w((z,x)) = w((z,y)) = 1$ and 
\[
w((0,z)) = \Big(\frac{2}{(2-2^p)^{q/p}} - 1\Big)^{1/q}.
\]
Consider $\NN = \{0,x,y\} = \leaf(T)\cup \{0\}\subset \MM_{T,q}$. By Corollary~\ref{cor:twoPointsFormula} we have
\[\begin{split}
\|\delta(x) + \delta(y)\|_{\F_p(\NN)}^p & = \min\{2d_{w,q}^p(x,0), 2^pd_{w,q}^p(x,0) + d_{w,q}^p(x,y)\}\\
&  = \min\{2(1 + w((0,z))^q)^{p/q},2^p(1 + w((0,z))^q)^{p/q} + 2^{p/q}\}\\
& = \min\{2\frac{2^{p/q}}{2-2^p},2^p\frac{2^{p/q}}{2-2^p} + 2^{p/q}\}\\
& = \frac{2^{p/q}}{2-2^p}\min\{2,2^p+2-2^p\} = \frac{2\cdot 2^{p/q}}{2-2^p}.
\end{split}\]
Moreover,
\[\begin{split}
T_{d_T,q}((1,1))^p & = 2^p w((0,z))^p + w((z,x))^p + w((z,y))^p = 2^p\frac{(2 - (2-2^p)^{q/p})^{p/q}}{2-2^p} + 2.
\end{split}\]
Thus, by Theorem~\ref{thm:reductionExamples} we have
\[\begin{split}
\Big(\absamen_p^q(2)\Big)^p & \geq \frac{\|\delta(x) + \delta(y)\|_{\F_p(\NN)}^p}{T_{d_{T,q}}((1,1))^p} = \frac{2\cdot 2^{p/q}}{2^p(2 - (2-2^p)^{q/p})^{p/q} + 2\cdot(2-2^p)}\\
& = \begin{cases} \frac{4}{4+ 2^p(2^p-2)} & \text{ if $q=p$},\\
\frac{2^{p+1}}{4 + 2^p\big((2 - (2-2^p)^{1/p})^p  - 2\big)} & \text{ if $q=1$}.
\end{cases}
\end{split}\]
\end{proof}

We note the issue of finding the value of $\absamen_p(2)$ is in a certain sense similar to the research considered in \cite{Basso22}, where G. Basso was considering the problem of finding for $n=3,4$ optimal value of constant $\operatorname{ae}(n)>0$ such that any metric space with $n$ points is complementably $1$-amenable in any superspace with constant $\operatorname{ae}(n)$ (he used different terminology, but it is easily seen that the problem he considered is really equivalent to what is mentioned above). In \cite{Basso22} using numerical simulations the author obtained a lower bound for $\operatorname{ae}(4)$ and using the solution of Gr\"unbaum conjecture \cite{ChL10}, whose prove is based again on numerical software, he was able to find an exact value of $\operatorname{ae}(3)$. It seems that the problem of finding exact value of $\absamen_p(n)$ for small numbers $n$ could be of a similar technical character, where the use of a numerical software could be quite helpful.

\section{Open questions}\label{sec:open}

The most important problem remains \cite[Question 6.2]{AACD}, that is, whether given $p$-metric spaces $\NN\subset\MM$, $\NN$ is $p$-amenable in $\MM$ (Theorem~\ref{thm:metricSolution} gives positive answer for metric spaces). By \cite[Lemma 2.5]{AACD20}, it is equivalent to ask whether we have $\sup_{n\in\Nat}\absamen_p(n)<\infty$. Even after a considerable effort (including numerical simulation, where we considered several millions of randomly generated $p$-metric spaces) we were not able to find a single example of $p$-metric spaces $\NN\subset\MM$ for which we would have $\amen_p(\NN,\MM)>2^{1/p}$. This leads us to ask the following.

\begin{question}
Given $0<p\leq q\leq 1$, is it true that $\sup_{n\in\Nat}\absamen_p^q(n)\leq 2^{1/q}$? Is it true at least for $q=p$ or for $q=1$?
\end{question}

Concerning the value of $\absamen_p(2)$ the best estimate we have is contained in Proposition~\ref{prop:estimateTwoPoints}. This leads us to ask e.g. the following two related problems, which seem to be of a very technical nature. We suspect the answer to the first one is positive, and the answer to the second one is negative (which would be interesting as it would imply that the problem of whether $\NN$ is isometrically $p$-amenable in $\MM$ for metric spaces $\NN\subset\MM$ with $|\NN|=3$, depends on the value of $p\in (0,1)$).

\begin{question}
Is it true that $\absamen_p(2) = \frac{4}{4+ 2^p(2^p-2)}$ for every $p\in (0,1)$?
\end{question}

\begin{question}
Is it true that $\absamen_p^1(2) > 1$ for every $p\in (0,1)$?
\end{question}

\end{document}